\documentclass[letterpaper, 11pt,  reqno]{amsart}

\usepackage{amsmath,amssymb,amscd,amsthm,amsxtra, esint, mathrsfs}
\usepackage{faktor}
\usepackage{tikz}

\usepackage{mathabx}

\usepackage{esvect}
\usepackage{enumerate}
\usepackage{bbold}

\usepackage[colorinlistoftodos,prependcaption,textsize=tiny]{todonotes}

\usepackage{mathtools}
\usepackage{color}
\usepackage[implicit=true]{hyperref}

\usepackage{cases}

\usepackage[left=32mm, right=32mm, 
bottom=27mm]{geometry}

\allowdisplaybreaks[2]

\usetikzlibrary{shapes.misc}
\usetikzlibrary{shapes.symbols}
\usetikzlibrary{shapes.geometric}

\tikzset{
	ddot/.style={circle,fill=white,draw=black,inner sep=0pt,minimum size=0.8mm},
	>=stealth,
	}

\tikzset{
	ddot2/.style={circle,fill=black,draw=black,inner sep=0pt,minimum size=0.8mm},
	>=stealth,
	}

\newtheorem{theorem}{Theorem} [section]

\newtheorem{lemma}[theorem]{Lemma}
\newtheorem{proposition}[theorem]{Proposition}

\theoremstyle{definition}
\newtheorem{definition}[theorem]{Definition}
\newtheorem{remark}[theorem]{Remark}
\newtheorem{assumption}{Assumption}

\theoremstyle{remark}
\newtheorem{example}[theorem]{Example}

\DeclareMathOperator*{\supp}{supp}

\newcommand{\Z}{\mathbb{Z}}
\newcommand{\R}{\mathbb{R}}

\newcommand{\T}{\mathbb{T}}

\newcommand{\1}{\mathbb 1}

\let\Re=\undefined\DeclareMathOperator*{\Re}{Re}

\DeclareMathOperator{\Lip}{Lip}

\let\P= \undefined
\newcommand{\P}{\mathbb{P}}

\newcommand{\E}{\mathbb{E}}

\def\norm#1{\|#1\|}

\newcommand{\dl}{\delta}

\newenvironment{renumerate}{\begin{enumerate}[\normalfont{(}i\normalfont{)}]}{\end{enumerate}}                                                                              

\newcommand{\Dl}{\Delta}
\newcommand{\eps}{\varepsilon}

\newcommand{\s}{\sigma}

\newcommand{\ft}{\widehat}

\newcommand{\wt}{\widetilde}
\newcommand{\cj}{\overline}

\newcommand{\dt}{\partial_t}

\newcommand{\dual}[2]{\left\langle#1,#2\right\rangle}

\newcommand{\les}{\lesssim}

\newcommand{\jb}[1]
{\langle #1 \rangle}

\newcommand{\N}{\mathbb{N}}

\renewcommand{\H}{\mathcal{H}}

\newtheorem*{ackno}{Acknowledgements}

\newcommand{\wick}[1]{{:}\,#1\mspace{2mu}{:}}

\numberwithin{equation}{section}
\numberwithin{theorem}{section}

\let\ve\vec
\let\vec\relax
\newcommand{\vec}[2]{\begin{pmatrix} #1 \\ #2 \end{pmatrix}}

\renewcommand{\u}{\mathbf u}
\renewcommand{\v}{\mathbf v}
\newcommand{\w}{\mathbf w}
\renewcommand{\d}{\mathrm d}

\newcommand{\U}{\mathbf U}

\newcommand{\loc}{\mathrm{loc}}

\DeclareMathOperator{\Law}{Law}

\newcommand{\Bo}{\mathscr B}
\newcommand{\G}{\mathscr G}

\renewcommand{\#}{\sharp}
\newcommand{\ph}{\varphi}

\renewcommand\labelenumi{\rm{(}\roman{enumi}\rm{)}}
\renewcommand\theenumi\labelenumi

\newcommand{\io}{\iota}

\begin{document}
\baselineskip = 15pt

\title[Ergodicity for the hyperbolic $P(\Phi)_2$-model]
{Ergodicity for the hyperbolic $P(\Phi)_2$-model}

\author[L.~Tolomeo]
{Leonardo Tolomeo}

\address{Leonardo Tolomeo\\
 Leonardo Tolomeo\\
School of Mathematics\\
The University of Edinburgh\\
and The Maxwell Institute for the Mathematical Sciences\\
James Clerk Maxwell Building\\
The King's Buildings\\
Peter Guthrie Tait Road\\
Edinburgh\\ 
EH9 3FD\\
 United Kingdom\\
and
 Mathematical Institute\\
 Hausdorff Center for Mathematics\\
 Universit\"{a}t Bonn\\
 Bonn\\
 Germany}

\email{l.tolomeo@ed.ac.uk}

%
%


\begin{abstract}
We consider the problem of ergodicity for the $P(\Phi)_2$ measure of quantum field theory under the flow of the singular stochastic (damped) wave equation \linebreak $u_{tt} + u_t + (1-\Delta) u + {:}\,p(u)\mspace{2mu}{:} = \sqrt 2 \xi$, posed on the two-dimensional torus $\mathbb T^2$. 
We show that the $P(\Phi)_2$ measure is ergodic, and moreover that it is the unique invariant measure for (the Markov process associated to) this equation which belongs to a fairly large class of probability measures over distributions.

The main technical novelty of this paper is the introduction of the new concepts of asymptotic strong Feller and asymptotic coupling {restricted to the action of a group}. We first develop a general theory that allows us to deduce a suitable support theorem under these hypotheses, and then show that the stochastic wave equation satisfies these properties when restricted the action of translations by shifts belonging to the Sobolev space $H^{1-\varepsilon} \times H^{-\varepsilon}$. We then exploit the newly developed theory in order to conclude ergodicity and (conditional) uniqueness for the $P(\Phi)_2$ measure.
\end{abstract}

%


\maketitle
\tableofcontents

\baselineskip = 14pt

\section{Introduction}
In this paper, we consider the (massive) $P(\Phi)_2$ measure of quantum field theory, formally given by 
\begin{equation}\label{Pphi2}
 dP(\Phi)_2(u) = \frac 1 Z \exp\Big(-\int_{\T^2} \wick{P(u)} - \frac 12 \int_{\T^2}  u (m^2 - \Delta) u \Big)du.  \tag{$P(\Phi)_2$}
\end{equation}
Here $P$ is a polynomial of even degree $2k \in \N$ with positive leading coefficient, i.e.\ $a_{2k} > 0$ and
$$ P(x) = a_{2k} x^{2k} + a_{2k-1} x^{2k-1} + \dotsb + a_0, $$
and $\wick{P}$ denotes the Wick renormalisation (which will be rigorously introduced in Section {3}). 
The construction of these measure has firstly been achieved by Guerra, Rosen and Simon in \cite{GRS}, and has been one of the major milestones in the program of constructive quantum field theory. 

In \cite{PW}, Parisi and Wu suggested a new approach to the construction of measures such as \eqref{Pphi2}, kickstarting the project of stochastic quantisation. 
In short, this project consists of the following. If we see \eqref{Pphi2} (formally) as a measure of the form
\begin{equation} 
 \sigma = \frac 1 Z \exp(-V(u)) du, \label{sigmaFD}
\end{equation}
then we can write the (overdamped) Langevin equation for such a measure, i.e.\ the stochastic differential equation (SDE)
\begin{equation} \label{sqel}
u_t =  -\nabla V(u) + \sqrt 2 \xi,
\end{equation}
where $\xi$ is a space-time white noise. In the finite dimensional setting, the measure $\sigma$ is invariant for \eqref{sqel}. If then one shows that the equation \eqref{sqel} admits a unique invariant measure, 
we can exploit this uniqueness to \emph{redefine} $\sigma$ as the unique invariant measure for \eqref{sqel}. From the point of view of numerics, this definition/program has the benefit that then one can generate samples of \eqref{Pphi2} by firstly solving \eqref{sqel} and then performing a Markov Chain Montecarlo (MCMC) procedure.

When specialised to the case of \eqref{Pphi2}, the equation \eqref{sqel} becomes the stochastic partial differential equation (SPDE)
\begin{equation} \label{SQE} \tag{SQE}
u_t =  -(m^2-\Delta) u - \wick{p(u)} + \sqrt 2 \xi,
\end{equation}
which in this context has the name of stochastic quantisation equation (SQE) for the measure \eqref{Pphi2}. Here $p(x) = P'(x)$ denotes the derivative of $P$.
To this day, the stochastic quantisation program for \eqref{SQE} has been successfully completed, with local well posdeness for \eqref{SQE} being shown by Da Prato and Debussche in \cite{DaDe}, global well posedness shown by Mourrat and Weber in \cite{MW} for $2k=4$ and by Tsatsoulis and Weber in \cite{TW} for higher values of $k$, and unique ergodicity being shown by Tsatsoulis and Weber in \cite{TW}.
We would like to remark here that the efforts of concluding the program of stochastic quantisation go well beyond the study of the measure \eqref{Pphi2} and its associated overdamped Langevin dynamics. Indeed, the Langevin dynamics for measures of the form
$$ d\Phi^4_d = \frac 1Z \exp\Big(-\frac14 \int_{\T^d} u^4 - \text{ renormalisation } - \frac 12 \int_{\T^d}  u (m^2 - \Delta) u \Big) $$
has been constructed in the whole subcritical regime thanks to the theory of regularity structures developed by Hairer and collaborators \cite{h14,BHZ,BCCH}. Moreover, thanks to the work by Hairer and Mattingly \cite{hm17}, we now possess a general theory for showing the strong Feller property for equations such as \eqref{SQE}. We will enter more in details about what this entails in Section 1.1, but as observed in \cite{hm17}, 
this property automatically implies uniqueness for the invariant measure as soon as an invariant measure with full support is known to exist. This is indeed the case in most situations in which the $\Phi^4_d$ measure can be constructed explicitly.

Despite the success of the project of stochastic quantisation, the choice of the overdamped Langevin dynamics as the model equation to sample the measure \eqref{Pphi2} is somewhat arbitrary. Indeed, in the context of sampling measures of the form of \eqref{sigmaFD}, one can consider many other models, with the only restriction being that the measure $\sigma$ is an ergodic measure for the flow, and a set of initial data whose flow will converge to $\sigma$ (in the sense of Birkhoff's ergodic theorem) is known. 
If we focus our attention to the finite dimensional setting, in recent years the following kinetic Langevin equation has attracted particular attention 
\begin{equation}\label{csqel}
\begin{cases}
u_t =  v, \\
v_t = -v -\nabla V(u) + \sqrt 2 \xi.
\end{cases}
\end{equation}
The unique invariant measure for \eqref{csqel} is given by 
\begin{equation} \label{sigma2FD}
 \frac{1}{Z'} \exp\Big(-V(u) - \frac{|v|^2}2\Big) du dv, 
\end{equation}
so one can sample the measure $\frac{1}{Z}\exp(-V(u))du$ by sampling the law of the first component of the solution of \eqref{csqel}. This procedure has the name of Halmitonian Montecarlo (HCM). 
It has numerically been observed that HCM converges faster than MCMC in many situations. While the author is not an expert in this field, we can refer the interested reader to \cite{EGZ} (and references within), which contains a rigorous justification of these faster convergence rates for a class of potentials $V$.

With this point in mind, it would be interesting to complete the project of stochastic quantisation for the analogous of \eqref{csqel}, which is given by 
\begin{equation}\label{SDNLW} \tag{SDNLW}
\begin{cases}
u_{tt} + u_{t} + (1-\Delta)u + \wick{p(u)} = \sqrt2 \xi, \\
(u(0), u_t(0)) = \u_0.
\end{cases}
\end{equation}
Here we fixed $m^2 = 1$ for simplicity of notation, and we will keep this choice for the rest of this paper.
In the context of stochastic quantisation, this is the so-called \emph{canonical stochastic quantisation} equation for the measure \eqref{Pphi2}. The invariant measure for this equation is (formally) given by 
\begin{align}\label{rhodef}
\rho(u,u_t) &= P(\Phi)_2(u) \otimes \mu_0(u_t) \\
& ``= \frac 1 Z \exp\Big(-\int_{\T^2} \wick{P(u)} - \frac 12 \int_{\T^2}  u (1 - \Delta) u \Big)\exp\Big(-\frac12 \int_{\T^2} u_t^2\Big) du du_t". \notag
\end{align}
The main result of this paper is the first step in the resolution of the stochastic quantisation program applied to equation \eqref{SDNLW}, which can be summarised in the following statement. 

\begin{theorem}\label{ergodicity2d}\label{uniqueness2d}
The measure $\rho$ as in \eqref{rhodef} is an ergodic meausure for the Markov semigroup generated by the flow of \eqref{SDNLW}. Moroever, for $0 < \eps \le \eps_k \ll 1$, the measure $\rho$ is the unique invariant measure belonging to the class
\begin{equation}\label{pintergrability2d}
W^{1}_{\wick{p}} := \Big\{\mu\in\mathcal P(H^{-\eps}\times H^{-\eps-1}): \int \|\wick{p(u)}\|_{\H^{-\eps}} d \mu(u,u_t) < \infty\Big\}.
\end{equation}
\end{theorem}

While the uniqueness part of the statement is conditional to the first moment of $\|\wick{p(u)}\|$ being finite, Theorem \ref{uniqueness2d} still suggests an algorithm for sampling the measure \ref{Pphi2} according to a HCM procedure. 
More specifically, one could pick a (random) initial data $\u_0\in \H^{1-\eps}$, and a sample of the noise $\xi$. We can then compute $\Phi_t(\u_0,\xi_j)$ by solving the equation \eqref{SDNLW}. We then consider the statistical average of $\|\wick{p(u)}\|_{\H^{-\eps}}$ at some large time $T\gg1$,
$$ [\|\wick{p}\|]_T = \frac{1}{M} \sum_{j=1}^M \frac{1}{T} \int_0^T \|\wick{p(\Phi_t(\u_0,\xi))}\|_{\H^{-\eps}} dt.  $$
Then if $ [\|\wick{p}\|]_T \le K$ for some (appropriately chosen) constant $K$, we ``accept" the sample and use the flow $\{\Phi_t(\u_0,\xi)\}_{t \ge 0}$ to study the measure \eqref{Pphi2}. 
Otherwise, we pick a different (randomly chosen) initial data $\u_0'$ and restart the procedure. Since the measure $\rho$ is absolutely continuous with respect to the following Gaussian measure, 
\begin{equation} \label{rho0}
d \rho_0(u,u_t) = \frac 1 Z \exp\Big(- \frac 12 \int_{\T^2}  u (m^2 - \Delta) u -\frac 12 u_t^2 \Big) du du_t,
\end{equation}
if one chooses the initial data $\u_0$ as a random sample of the measure $\rho_0$, then the quantity $[\|\wick{p}\|]_T$ is going be finite almost surely, and in principle there should be no need to sample a different initial data $\u_0'$. However, the techniques of this paper cannot exclude the situation in which the solution ``escapes" the invariant measure due to numerical errors.

\subsection{The strong Feller property and its failure for stochastic wave equations}
It turns out that from a technical point of view, the stochastic quantisation program for wave equations such as \eqref{SDNLW} is much harder to achieve than in the parabolic case. 
While important milestones for the local well posedness theory for canonical stochastic quantisation equations are progressively been achieved, global well posedness and ergodicity results are both very rare. Indeed, local well posedness for equation \eqref{SDNLW} has been proven in \cite{gko} on $\T^2$, in \cite{TowaveR2, OTWZ} on $\R^2$, and in \cite{OTR} on a general 2-dimensional compact manifold. A series of 3-dimensional results have been proven in \cite{gko2} and \cite{OOT2} for the equation with quadratic nonlinearity, 
in \cite{OOT1,bring1,OWZ} for cubic nonlinearities under the addition of some smoothing in the equation, and finally in \cite{BDNY} for the canonical stochastic quantisation equation for the $\Phi^4_3$ measure. 

When an invariant measure is available, often global well-posedness for a.e.\ initial data sampled according to the invariant measure follows via an application of Bourgain's invariant measure argument \cite{Bo94} (see also \cite[Theorem 6.1]{FoTo} for a general formulation). 
However, this is more-or-less the only globalisation argument that has been shown to work for \emph{singular} stochastic wave equations. The only exceptions that the author is aware of, in which it is actually possible to show some appropriate (pathwise) energy estimates for the solutions, are the results by the author \cite{TowaveR2} and Gubinelli, Koch, Oh and the author \cite{gkot}.
While in principle good energy estimates are not necessary in order to prove (unique) ergodicity, they are a fundamental tool in many applications. Nevertheless, the main difficulty in showing ergodicity for stochastic wave equations comes from a different issue, which is the failure of the strong Feller property.

Before we can discuss in details the techniques developed in this paper to get around this problem, it is instructive to move back to the case of parabolic SPDEs (such as \eqref{SQE}), and describe the general strategy to proving unique ergodicity in that case. 
Let us denote the solution of a SPDE at time $t \ge 0$ with initial data $u_0$ and driven by a noise $\xi$ by $\Phi_t(u_0,\xi)$. Suppose that on an appropriate space $X$ of initial data, the solution $\Phi_{t}(u_0,\xi)$ exists for every $t > 0$.
Under reasonable assumptions on the equation, the noise $\xi$, and the local well posedeness theory, this defines a Markov process on the space $X$. In particular, for any bounded measurable function $F: X \to \R$, we can define 
$$ P_t F := \E[ F(\Phi_t(u_0,\xi)) ], $$
and $P_t$ will be a Markov semigroup. We denote its dual by $P_t^*$. Unique ergodicity for the SPDE then corresponds to having a unique invariant measure for the semigroup $P_t$. Typically, the main ingredients for showing such a statement are the following.
\begin{description}
\item[(Long time estimates)] Show good long time estimates for the flow, i.e.\ estimates of the form 
$$ \E\| \Phi_t(u_0,\xi) \| \le C(u_0) \text{ for every } u_0\in X, $$
where $\|\cdot\|$ is some appropriate norm of the solution\footnote{There is actually no need for this to be a norm, and one can consider situations in which $\|\cdot \|$ is replaced by an appropriate ``size function" $r: X \to \R$ such that $r(\Phi_t(u_0,\xi)) \to \infty$ as $t\uparrow t^*$ implies blowup at time $t^*$ (or more typically, is the \emph{definition} of blowup at time $t^*$). }, and the constant $C(u_0)$ is allowed to depend on the initial data $u_0$, but not on time.
\item[(Irreducibility)] Fix a base point $u_*$, and show that for every small ball $B_\eps(u_*)$, for every $R > 0$, and for every $u_0$ belonging to $B_R(u_*)$, we have that  
$$ \sup_{t \ge 0} \P( \{\Phi_t(u_0,\xi) \in B_\eps\} ) \ge 2\eps_0(\eps,R) > 0. $$
\item[(Coupling)] Show that for every $\delta > 0$ and for some appropriate distance on probability measures $d$, there exists some $\eps > 0$, such that for every $u_0 \in B_\eps(u_*)$, 
$$ \limsup_{t \to \infty} d( P_t^* \delta_{u_0}, P_t^* \delta_{u_*}) \le \delta.$$
\end{description}
The proof of (unique) ergodicity then goes roughly as follows: starting from $u_0$, by irreducibility, after some time $t_1 > 0$, we have that $\P(\Phi_{t_1}(u_0,\xi) \in B_\eps) \ge \eps_0(\eps, \|u_0\|)$. Then from the coupling property, the evolution starting from the ball $B_\eps$ will be ``close in law" to the evolution of ${u_*}$. For the part of the evolution that at time $t_1$ is outside the ball $B_\eps$, we repeat the same process: after a (random) time $t_2$, we have that $\P(\Phi_{t_2}(u_0,\xi) \in B_\eps| \Phi_{t_1}(u_0,\xi) \not \in B_\eps) \ge \eps_0(\eps, \|\Phi_{t_1}(u_0,\xi)\|)$. Iterating this process, we obtain that 
$$ \P(\{ \Phi_{t_j}(u_0,\xi) \not \in B_\eps \text{ for every } j\le J) \le \prod_{j=1}^J  (1 - \eps_0(\eps, \|\Phi_{t_{j-1}}(u_0,\xi)\|)).$$
This is where the long time estimates come into play: up to possibly extending the times $t_1, t_2, \dotsc$, we can guarantee that $\|\Phi_{t_{j-1}}(u_0,\xi)\|$ remains under control, and so $\sum_{j} \eps_0(\eps, \|\Phi_{t_{j-1}}(u_0,\xi)\|)$ diverges. Therefore, we obtain that  the evolution starting from $u_0$ is ``close in law" (with respect to some appropriate distance) to the evolution starting from $u_*$ with high probability, and this allows us to conclude uniqueness of the invariant measure. 
If one then has better control over the quantities $C(u_0), \eps_0(\eps,R),$ and $\eps$ as a function of $\delta$, it is also possible to extract a convergence rate from this argument, which will typically be exponential. 
It is also possible to slightly weaken the three properties above, at the cost of getting a worse convergence rate. 

In practice, however, one needs not to perform this complicated analysis and show tight control over the various quantities, but there are instead a number of pre-confectioned results that can be used in order to obtain exponential convergence to equilibrium. A classical example of such results is Harris theorem (see \cite{hm08} for a satisfying proof of this result). In this context of ``pre-pacakged" results, a very successful approach has been to rely on the \emph{strong Feller property}. In short, we say that a semigroup has the strong Feller property if, for some time $t> 0$, 
$$ F \text{ measurable and bounded} \Rightarrow P_t F \text{ is continuous.} $$ 
This is essentially an infinitesimal smoothing property of the semigroup $P_t$, and in the case of finite dimensional systems of SDEs, checking this property is in most situation a simple consequence of H\"ormander's hypoellipticity theorem. 
It is easy to check that, in the finite dimensional setting, both \eqref{sqel} and \eqref{csqel} always satisfy H\"ormander's condition, hence under extremely general assumptions on the potential $V$, the associated semigroups do possess the strong Feller property. Since the invariant measures \eqref{sigmaFD}, \eqref{sigma2FD} trivially have full support, the strong Feller property automatically implies unique ergodicity (see \cite[Corollary 3.9]{hm17}).
This approach has had incredible success in infinite dimension as well, with the result in \cite{hm17} showing that an incredibly large class of parabolic SPDEs has the strong Feller property. 

Even in situations where the strong Feller property fails or it is otherwise hard to prove, we do possess an alternative theory, developed by Hairer and Mattingly in their seminal work \cite{hmasf}. 
In their work, they introduced the notion of ``asymptotic strong Feller property". Morally spealing, this property does not require $P_t F$ to be continuous for any positive time $t > 0$, but requires the continuity to hold in the limit $t \to \infty$.\footnote{For the actual definition, we refer the reader to \cite[Section 3.2]{hmasf}.} The strength of both the strong Feller property and the asymptotic strong Feller property is encapsulated in the following support theorem.
\begin{proposition}[Theorem 3.16 in \cite{hmasf}] \label{suppclassic} Suppose that the semigroup $P_t$ has the (asymptotic) strong Feller property. Let $\mu, \nu$ be two invariant measures with $\mu \perp \nu$. Then we have 
$$\supp(\mu)\cap \supp(\nu) = \emptyset. $$
\end{proposition}
\noindent
It is fairly easy to see that the conclusion of this proposition and the irreducibility property above are in contradiction. Therefore, we have that irreducibility and the (asymptotic) strong Feller property imply unique ergodicity, if at least one invariant measure is shown exists.\footnote{More precisely, without this extra assumption, they imply that there exists \emph{at most} one invariant measure.} 
The combined toolbox of strong Feller property and asymptotic strong Feller property has been extremely successful in showing ergodicity for several classes of SPDEs, including the above mentioned results for singular parabolic SPDEs \cite{TW,hm17}, Navier-Stokes equations \cite{FM95, hmasf, HM3}, and in many situations with degenerate noise \cite{HM2b, BKS, BW, GHMR17}. 

However, such techniques do not seem to be easily applicable to wave equations, and more in general dispersive SPDEs. 
Indeed, the only results known to the author that prove ergodicity for stochastic dispersive PDEs are \cite{BdP, bos16, DO05, GHMR2, terg, FoToerg}. While some advances towards the low-regularity regime has been achieved in the results by the author and by Forlano and the author \cite{terg,FoToerg}, none of these results can deal with the singular regime.

The main observation to explain this discrepancy is that dispersive (stochastic) PDEs seem to \emph{never} have the strong Feller property on a connected state space, as it was firstly observed by the author in \cite{terg}. 
The reason is the following. By expressing the solution of a stochastic dispersive PDE using the Duhamel/variation of constants formula, we obtain that  
$$u(t) = S(t) u_0 + \sqrt2 \int_0^t S(t-t') \xi(t') dt' +  \mathcal N_t(u), $$
where $S(t)$ denotes the linear propagator for the equation at hand, $\xi$ is the particular choice of the noise and $\mathcal N_t(u)$ denotes a nonlinear remainder. 
We expect the regularity of the solutions (and hence the space in which the invariant measures are concentrated) to be dictated by the stochastic convolution  $\psi = \sqrt2 \int_0^t S(t-t') \xi(t') dt' ,$ which will typically belong to some Sobolev space $H^{s_0-\eps}$ for some $s_0 \in \R$, and every $\eps >0$, \emph{but not} belong to $H^{s_0}$. This forces us to take a state space $X$ for the Markov semigroup to be rich enough to contain functions that belong to $H^{s_0-\eps} \setminus H^{s_0}$, but not are not any smoother. 
As it is common in this business, we expect the nonlinear remainder to be smoother than the linear solution.\footnote{As far as the author is aware, this is the case for every (stochastc) dispersive equation that has a satisfactory local well posedness theory.} 
The main difference with the parabolic case, is that the linear propagator is \emph{invertible} in the Sobolev spaces $H^\sigma$ for every $\sigma \in \R$. In particular, the linear propagator preserves the regularity of the initial data. 
Therefore, one can test the definition of the strong Feller property on the following indicator functions
\begin{equation}
 F(u) = \1_{H^{s_0-\eps} \setminus H^{s_0}}(u). \label{nonSF}
\end{equation}
From the discussion abobe, one obtains that 
$$ P_t F(u) = F(S(t)u) = F(u), $$
which is also an indicator function, hence it is not continuous (as long as the space is connected, and there exists at least one element of the state space that also belongs to $H^{s_0}$). 

In principle, one could try to shrink the state space in order to avoid this kind of counterexamples, but in many applications this seems to be a fool's errand.
The reason is that one can replicate the counterexample above by replacing $H^{s_0-\eps} \setminus H^{s_0}$ with any set $S$ which is both invariant for the linear propagator $S(t)$ and by (relatively) smooth perturbations. In the case of dispersive equations, this is an extremely rich family.

Of course, as per the discussion above, one could try to completely avoid relying on the strong Feller property, and instead attempt building the theory using the asymptotic strong Feller property instead (or the related concept of asymptotic couplings, see \cite{HM3, GHMR17, BKS}). In principle, this seems to be a reasonable approach, since the linear equation associated to \eqref{SDNLW}
$$ \begin{cases}
u_t = v \\
v_t = -v + (1-\Delta)u
\end{cases}
$$
has a propagator $S(t): \vec{u}{v} \mapsto S(t) \vec{u}{v}$ which satisfies 
$$ \|S(t)\|_{H^s\times H^{s-1} \righttoleftarrow} \les e^{-\frac t2}. $$
However, in order to show such property, one would need a good long-time estimate on the difference of two solutions $\Phi_t(\u_0,\xi_1) - \Phi_t(\u_1,\xi_2)$, where $\xi_1,\xi_2$ are two copies of the noise with the same law as $\xi$.\footnote{One actually just needs some kind of control of the signed measures $\Law(\xi_j) - \Law(\xi)$ in total variation. This generalisation is extremely useful in many applications, and we will make use of similar ideas in the following sections.}
This is how Forlano and the author achieved the ergodicity result for $2k=4$ in the non-singular case \cite{FoToerg}. However, dealing with the singular case seems to be beyond the current technology. 
The reason is that, since the nonlinearity $u^3$ has a controlled modulus of continuity only on bounded set, one would need to show some good global estimates for a \emph{single} solution $\Phi_t(\u_0,\xi)$ to begin with. 
As discussed earlier in the introduction, this seems to be extremely hard in the singular case, and such an estimate is not known for any singular wave equation\footnote{The estimates in \cite{TowaveR2, gkot} grow with a double exponential in time, which is way too fast for this argument to work - and their proof works only in the case $2k=4$.}. 
The main reason is that, contrarily to the parabolic case, the only ``useful" quantity to control the global evolution for wave equations is the energy
$$ E(u,u_t) = \frac{a_{2k}}{2k} \int u^{2k} + \frac12 \int |\nabla u|^2 + \frac12 \int u^2 + \frac 12 \int |u_t|^2, $$
up to some refinements. It should not be a surprise that when solutions became rougher and rougher, the quantity above gives progressively less and less information on the growth of solutions, up to a point in which the argument breaks down completely. 
It is interesting that for $2k=4$, the threshold of regularity for obtaining energy estimates corresponds exactly to the threshold for singularity of the equation \eqref{SDNLW}. 
It is unclear to the author if this is just an accident due to the proof techniques, or there is a deeper connection between the two. \label{sec:sf}
\subsection{Asymptotic couplings restricted to the action of a group} 

In view of the discussion in Section \ref{sec:sf}, one might wonder how it is possible to get a positive result in Theorem \ref{uniqueness2d} to begin with. The starting point is the following observation by the author in \cite{terg}, that now we adapt to \eqref{SDNLW}. By the analysis in \cite{gko,gkot}, we know that the solution of \eqref{SDNLW} can be written as 
$$ \u(t) = S(t) \u_0 + \ve{\psi}(t, \xi) + \v(t), $$
where $S(t) \u_0 + \ve{\psi}(t, \xi)$ denotes the (vector) solution $(\psi,\psi_t)$ to the \emph{linear} equation
\begin{equation} \label{SDLW}
\begin{cases} 
\psi_{tt} + \psi_t = \sqrt2 \xi, \\
(u,u_t)(0) = \u_0,
\end{cases} 
\end{equation}
and $\v(t) \in H^{1-\eps} \times H^{-\eps} =: \H^{1-\eps}$ is a smoother nonlinear reminder. For simplicity of notation, denote $\H^s := H^{s} \times \H^{s-1}$. Inspired by the form of the functionals \eqref{nonSF} that were used to disprove the strong Feller property, we consider  
$$ F(\u) = \1_{\H^{-\eps} \setminus \H^{1-\eps}}. $$
We remark here that $\H^{-\eps} \setminus \H^0$ is the typical regularity of solutions with initial data sampled according to \eqref{rhodef}. As discussed in Section 1.1, we have that 
$$ P_tF(\u_0) = F(\Phi_t(\u_0,\xi)) = F(S(t)\u_0) = F(\u_0).$$
As a consequence, $P_tF$ is not continuous in the topology of $\H^{-\eps} $. However, if $\u_0^1, \u_0^2$ are such that $\u_0^1 - \u_0^2 \in \H^{1-\eps} $, then we have 
$$ P_tF(\u_0^1) - P_tF(\u_0^2) = 0. $$
In particular, if we consider the distance 
$$ d_{\H^{1-\eps}}(\u_0^1, \u_0^2) = \| \u_0^1 - \u_0^2 \|_{\H^{1-\eps}}, $$
one has that $P_tF$ is \emph{continuous} in the topology induced by $d_{\H^{1-\eps}}$. 
While in this way the state space is not connected, the change of topology removes the main obstruction to  showing the strong Feller property. 
If then we are able to show a (stronger) version of Proposition \ref{suppclassic}, we can deduce that if two invariant measures $\nu_1, \nu_2$ are such that $\nu_1 \perp \nu_2$, then there exists a set $E$ such that 
$$ \nu_1(E) = 1, \hspace{10pt}  \nu_2(E) =0, \hspace{25pt} E = E + \H^{1-\eps} $$
Finally, one can show that such properties are in contradiction with the extra property $\nu_1, \nu_2 \ll \rho$. From this, we deduce that the measure $\rho$ must be ergodic. 
The reason why the contradiction holds, is that the family $\{E = E + \H^{1-\eps}\}$ is contained in the $\sigma$-algebra generated by sure events for the Gaussian measure $\rho_0$ (see Lemma \ref{0-1} and \cite[Remark 5.8]{terg}), and so for any such set $E$ we must have $\nu_j(E) = \rho_0(E)$. 

The strategy described above is essentially how ergodicity for \eqref{SDNLW} when $2k=4$ was shown on the one dimensional torus $\T$ in the previous work by the author \cite{terg}. However, there are a series of issues in extending this strategy to the 2 dimensional case.
The most important of these, is that it is unclear if the strong Feller property holds after the change of topology induced by $d_{\H^{1-\eps}}$. 
The technical reason for this, is the fact that the space $\H^{1-\eps}$ is strictly bigger than the Cameron-Martin space for the Gaussian measure $\rho_0$, as opposed to the analogous property for the 1-dimensional flow, which does hold. This element was used crucially in \cite{terg}. 

Nevertheless, we can still show the following coupling property: for every $\u_0$ such that the solution $\Phi_t(\u_0,\xi)$ does not grow too fast, and for every $\u_0^2$ such that $\u_0 - \u_0^2 \in \H^{1-\eps}$, there exist a noise $\xi'$ which satisfies $\Law(\xi') \ll \Law(\xi)$ and 
$$ \P\big(\{\| \Phi_t(\u_0^1, \xi) - \Phi_t(\u_0^2, \xi') \|_{\H^{1-\eps}} \les e^{-\frac t4}\} \big) > \frac12.   $$
This is essentially the content of Lemma \ref{lemma:hdef} below, and can be shown via a carefully chosen Girsanov shift argument.  While we cannot show that the growth assumption on $\Phi_t(\u_0,\xi)$ holds for every initial data $\u_0$ due to the difficulties described in Section 1.1, 
we can show that it holds almost surely according to any invariant measure in the class $W^1_{\wick{p}}$. The main novelty of this paper then consists in codifying the correct support theorem that holds under the very weak coupling assumption above. 

To this scope, it is convenient to put everything into an abstract framework, and show a support theorem that holds in a general setting. We consider the space $Y$ of ``good" initial data (such that the solution has controlled growth as $t \to \infty$), and we define the state space to be
$$ X = Y + \H^{1-\eps} \subset \H^{-\eps}. $$
Then we see the space $\H^{1-\eps}$ as a group $\G$ acting on $X$ by translations, i.e.\ 
$$ \G \times X \ni (\v_0, \u_0) \mapsto \tau_{\v_0}(\u_0) = \u_0 + \v_0. $$
This action allows us to define a new topology on the space $X$, which is induced by the distance 
$$ d_\G(\u_1,\u_2) = \inf \{ 1\wedge \|\v_0\|_{\H^{1-\eps}} :  \tau_{\v_0}(\u_1) = \u_2\} = 1\wedge \|\u_1 - \u_2\|_{\H^{1-\eps}}  $$
Under this distance, the space $X$ loses most of the ``good" measure-theoretical properties of Polish spaces, namely, the space $(X, d_\G)$ is not separable, it has uncountably many connected components, 
and the invariant measure $\rho$ will not be a Radon measure on this space. Nevertheless, thanks to this definition, the coupling property above can be codified in the following way: for every $\u_0 \in X$, for every $\v_0 \in \G$ and for every $\ph: (X,d_\G) \to \R$ Lipschitz, we have that 
$$ \big| P_{t} \ph(\u_0) - P_{t} \ph(\tau_{\v_0}(\u_0)) \big| \le 2\eps_0(\u_0,\v_0) \|\ph\|_{\infty} + e^{-\frac t4} C(\u_0,\v_0) \|\ph\|_{\Lip}
 $$
for some $\eps_0 < 1$. This is the basis for the concept of \emph{asymptotic coupling property restricted to the action of a group} that we will introduce in Section 2. Similarly, when $\eps_0 \to 0$ as $\|\v_0\| \to 0$, we say that the semigroup has the asymptotic \emph{strong Feller} property restricted to the action of the group $\G$. 
One needs to be very careful with measurability issues here, since for a function $\ph: (X,d_\G) \to \R$, being Lipschitz does not automatically imply being measurable. Nevertheless, we can still deduce a support theorem in the guise of Proposition \ref{suppclassic}, which we formulate in Theorem \ref{supp1} and Theorem \ref{supp2}. 
The latter states that under an appropriate asymptotic coupling assumption, if $\mu, \nu$ are two invariant measures with $\mu \perp \nu$, and $X/ \G$ is the space of orbits of the action of $\G$, then $\pi_\# \mu \perp \pi_\# \nu$ as well, where $\pi: X \to X/\G$ is the canonical projection. Together with properties of the linear evolution for the equation \eqref{SDLW}, this support theorem allows us to conclude the result of Theorem \ref{uniqueness2d}.
\subsection{Structure of the paper}
\begin{description}
\item[Section 2] In this section, we introduce the abstract theory and prove the main support theorems (Theorem \ref{supp1} and \ref{supp2}). More specifically, in the various subsections we will do the following.
\begin{enumerate}
\item[2.1.] We introduce the assumptions on the measurable space $X$, the semigroup $P_t$ and the group action $\tau: \G \times X \to X$.
\item[2.2.] We introduce the definitions of the concepts of asymptotic strong Feller property restricted to the action of a group (Definition \ref{rASFdef}) and asymptotic coupling property restricted to the action of a group (Definition \ref{rACdef}). 
\item[2.3.] We state our main support theorems, Theorem \ref{supp1} and Theorem \ref{supp2}, and in Example \ref{ex:ac->rac} and Remark \ref{rk:weakness}, we discuss how they relate to the existing theory.
\item[2.4.] We perform the proof of the main theorems of this section. 
\end{enumerate}
\item[Section 3] In this section, we focus our attention to the hyperbolic $P(\Phi)_2$ model, and perform the proof of Theorem \ref{uniqueness2d}. More specifically, in the various subsections we will do the following.
\begin{enumerate}
\item[3.1.] We rigorously define the $P(\Phi)_2$ measure and the Wick renormalisation, and discuss a number of properties of each that are relevant for the proof of Theorem \ref{uniqueness2d}. 
\item[3.2.] We collect the existing local and global theory for equation \eqref{SDNLW}, and use them to build a Markov process on the space $\H^{-\eps}$.
\item[3.3.] We define the space $Y$ of ``good" initial data $\u_0$, and via a Girsanov shift argument, show the main estimates conclusive to the coupling property for \eqref{SDNLW}.
\item[3.4.] We build the Markov semigroup $P_t$ on the space $X = Y + \H^{1-\eps} \cup \{\infty\}$, and show that it satisfies the asymptotic coupling property restricted to the action of $\H^{1-\eps}$ on $X$. We also show that the asymptotic strong Feller property restricted to the action of $\H^{1-\eps}$ holds on $Y$.
\item[3.5.] We show the $0-1$ property for the measure $\rho_0$, and combine it with Theorem \ref{supp2} to deduce that the $P(\Phi)_2$ measure $\rho$ must be ergodic.
\item[3.6.] We focus our analysis to the class $W^1_{\wick{p}}$, and show that for every invariant measure $\mu$ in this class, one must have that $\pi_\# \mu = \pi_\# \rho$. We then use this, together with the support theorem \ref{supp2}, to conclude uniqueness in the class $W^1_{\wick{p}}$ and hence the proof of Theorem \ref{uniqueness2d}.
\end{enumerate}

\end{description}

\subsection{Further remarks}
\begin{remark}
The main part of the analysis for equation \eqref{SDNLW} will happen on the state space $X$ defined in \eqref{Xdef2d}. 
Unfortunately, at this stage we are not able to show that this space is a Borel subset of $\H^{-\eps}$ (even if we believe it should be), but only that it has full measure according to any invariant measure in $W^1_{\wick{p}}$. 
While this seems to be a minor point, due to the unusual setting of the theory in Section 2, we choose to take an extremely cautious approach, 
and slowly check that  by performing all the straightforward modifications to the definition of the measures and of the semigroup, one is still able to apply the theory of Section 2 to conclude Theorem \ref{uniqueness2d}. This is the role of the map $\iota^\ast$ introduced in Section 3.5. Unfortunately, the addition of this map makes the proofs of Section 3.5 and 3.6 notationally heavy. On a first reading, the author would suggest that the reader assumes that the set $X$ is a Borel subset of $\H^{-\eps}$, in which case $\iota^\ast$ is simply the identity map (after restricting the measure to $X$).
\end{remark}

\begin{remark} \label{Sobolevreg}
In the result of Theorem \ref{uniqueness2d}, the particular choice of $\eps_k$ is such that the Sobolev inequality 
$$ \|u^{2k-1}\|_{H^\eps(\T^2)} \les \| u\|_{H^{1-\eps}}^{2k-1}, $$
holds, and does not play a major role in the proof (see Lemma \ref{2dcompatibility}). It is likely possible to push the value of $\eps$ to $\eps < \min(\frac{2}{2k-2}, \frac14 + \frac{1}{2k-2})$, which corresponds to the local well posedness theory for the equation \eqref{SDNLW} (see \cite{gko}). However, such an extension would make the technical part of the proof significantly harder to digest, without substantially affecting the result of Theorem \ref{uniqueness2d}. For sake of exposition, we decided to avoid this further complication.
\end{remark}
\begin{remark}
The (conditional) uniqueness result in Theorem \ref{uniqueness2d} depends on the particular definition of the flow $\Phi_t(\u_0, \xi)$ of \eqref{SDNLW}. 
More precisely, whenever the initial data $\u_0$ of \eqref{SDNLW} does not satisfy the hypotheses of Theorem \ref{LWP}, the equation does not admit a satisfying local well posedness theory, and so the definition of the flow starting from these data is somewhat arbitrary. The choice that we make in this paper is to declare that if $\u_0$ does not satisfy such hypotheses, then the flow ``blows up immediately", which prevents the existence of pathological invariant measures which are concentrated on a set where the flow is not well defined. See \eqref{flowdef} for the precise definition.
\end{remark}
\begin{ackno}
The author wishes to thank Tadahiro Oh for his continuous encouragement and support during the preparation of this paper. 

The author was partially supported by the Deutsche
Forschungsgemeinschaft (DFG, German Research Foundation) under Germany's Excellence
Strategy-EXC-2047/1-390685813, through the Collaborative Research Centre (CRC) 1060. 
\end{ackno}

\section{Restricted asymptotic strong Feller and restricted asymptotic coupling properties}\label{theory}
In this Section, we introduce the abstract concepts of asymptotic strong Feller and asymptotic coupling restricted to the action of group, and show how they imply a support theorem in the same vein as Proposition \ref{suppclassic}. The main results of this section are Theorem \ref{supp1} and Theorem \ref{supp2}.
\subsection{Assumptions}
Throughout this section, we will assume the following.
\begin{assumption}\label{ass:metric}
Let $X$ be a metric space, with distance $d_X$. Let $\Bo(X)$ be its Borel sigma-algebra, and let 
$$ \mathscr L^\infty(X) := \{ f : X \to \R \text{ Borel}:  \sup_{x \in X} |f(x)| < \infty\},  $$
equipped with the sup-norm, denoted by $\|f\|_{\infty}$. 
\end{assumption}
\begin{assumption}\label{ass:markov}
It is given a Markov semigroup $(P_t)_{t \ge 0}$ on $\mathscr L^\infty(X)$. More precisely,
\begin{enumerate}
\item For every $t \ge 0$, $P_t: \mathscr L^\infty(X) \to \mathscr L^\infty(X)$ is linear and bounded.
\item For every $t,s \ge 0$, $P_{t+s} = P_t P_s = P_s P_t$.
\item For every $t \ge 0$ and for every $f$ Borel with $f(x) \ge 0$ $\forall x \in X$, then $P_t f (x) \ge 0$ $\forall x \in X$.
\item Denoting by $\1$ the constant function $\1(x) = 1$ $\forall x \in X$, for every $t \ge 0$ we have that $P_t\1 = \1$. 
\end{enumerate}
\end{assumption}
\begin{assumption}\label{ass:group}
We have a topological group $\G$ with identity $e$, whose topology is induced by a left-invariant distance. More precisely, there exists a function $|\cdot|: \G \to \R$ such that 
\begin{enumerate}
\item for every $g \in \G$, $|g| \ge 0$, and $|g|=0$ if and only if $g = e$,
\item for every $g \in \G$, $| g^{-1}| = |g|$,
\item for every $g_1, g_2 \in \G$, 
$$ |g_1g_2| \le |g_1| + |g_2|, $$
\end{enumerate}
and the distance $d_\G(g_1,g_2)$ between two elements $g_1, g_2$ is given by  
$$ d_\G(g_1,g_2) = |g_1^{-1} g_2| = |g_2^{-1} g_1|.  $$
Given $r \ge 0$, we denote by $B_r$ the closed ball with centre $e$ and radius $r$:
$$ B_r := \{ g \in \G: |g| \le r\}. $$
\end{assumption}
We remark that Assumption \ref{ass:group} holds if and only if the group $\G$ is metrisable. See \cite[Theorem 8.3]{HeRo}.
\begin{assumption}\label{ass:action}
We have a group action $\tau: \G \times X \to X$. More precisely, by denoting $\tau_g(u) := \tau(g,u)$, we have that 
\begin{enumerate}
\item for every $x \in X$, $\tau_e(x) = x$,
\item for every $g_1,g_2 \in \G, x \in X$, $\tau_{g_1g_2}(x) = \tau_{g_1}(\tau_{g_2}(x))$. \vspace{3pt}
\end{enumerate}
\end{assumption}
\begin{assumption}\label{ass:tautopology}
For every compact set $K \subseteq X$, the map $g \mapsto \tau_g(x)$ is equicontinuous in $e$ for $x \in K$. More precisely, 
\begin{equation}\label{equicont}
 \lim_{r \to 0} \sup_{x \in K} \sup_{g \in B_r} d_X(x,\tau_g(x)) = 0. 
\end{equation}
Moreover, for every $r \ge 0$ and for every compact set $K \subset X$, we have that 
\begin{equation} \label{measurability}
 \tau(B_r \times K) \in \Bo(X).
\end{equation}
\end{assumption}
We note that the assumptions above are very general, and we do not require many of the usual properties of the space $X$ and the action $\tau$. For instance, the space $X$ does not need to be complete, nor separable, and the action $\tau$ does not need to be continuous in the $X$-variable. 
The only compatibility conditions between the topologies of $X$ and $\G$ respectively is delineated in Assumption \ref{ass:tautopology}. 
However, we will only work with Radon probabilities, which recovers a number of the usual properties of measures on Polish spaces.

The goal of being so general in the settings is not (only) being able to provide the most comprehensive statement possible. 
In order to obtain the result in Theorem \ref{uniqueness2d}, we will need to consider a space $X$ which is merely a subset of the Banach space $\H^{-\eps}$, without any clear connection with the topology of $\H^{-\eps}$ 
(and actually, it is not even clear if the state space is going to be a Borel subset of $\H^{-\eps}$). 
This prevents us from exploiting most of the ``usual" assumptions on the space $X$. While the group action we will consider is going to be continuous with respect to the topology of $\H^\eps$, it is convenient to allow for discontinuous actions in order to relate the results of this section with the existing theory, hence our choice of ``minimal" compatibility conditions in Assumption \ref{ass:tautopology}. 
See Example \ref{ex:ac->rac} for more details.

\subsection{Definitions} As discussed in Section 1.2, the group action will induce a new distance on the space $X$.
\begin{definition}
The action $\tau$ allows us to define another distance on the space $X$, that with a slight abuse of notation, we denote with $d_\G$. For $u_1, u_2 \in X$, we define $d_\G: X^2 \to [0, \infty]$~by
$$ d_{\G}(u_1,u_2) := \inf \{ |g|: g \in \G, u_1 = \tau_g(u_2) \}. $$
It is easy to check that this is indeed a metric (here \eqref{equicont} guarantees that $d_\G(u_1,u_2) = 0$ implies $u_1 = u_2$). For a function $\ph: X \to \R$, we say that $\ph$ is \emph{$\G$-Lipschitz} if there exists a constant $c \ge 0$ such that 
$$ |\ph(u_1) - \ph(u_2)| \le c d_\G(u_1,u_2), $$
and we denote 
$$ \|\ph\|_{\G-\Lip}:= \sup_{u_1,u_2 \in X, u_1 \neq u_2} \frac{|\ph(u_1) - \ph(u_2)|}{d_\G(u_1,u_2)}.$$
\end{definition}

We are now ready to introduce the main new properties in this work, the asymptotic strong Feller and asymptotic coupling properties restricted to the action of $\G$.
\begin{definition}\label{rASFdef}
We say that $P_t$ has the \emph{asymptotic strong Feller} property \emph{restricted to the action of $\G$} on a set $S \subseteq X$, in short $\mathrm{(rASF)_S}$, 
if there exists a sequence of times $t_n \ge 0$ and a sequence of positive real numbers $\delta_n \to 0$ such that for every $u_0 \in S$ and every $\ph \in \mathscr L^\infty(X)$ with $\|\ph\|_{\G-\Lip} < \infty$, we have 
\begin{equation}\tag{rASF} \label{rASF}
\big| P_{t_n} \ph(u_0) - P_{t_n} \ph(\tau_g(u_0)) \big| \le 2 \eps_0(u_0,g)  \|\ph\|_{\infty} + \delta_n C(u_0,g) \|\ph\|_{\G-\Lip},
\end{equation}
where $C(u_0,g) < \infty$ for every $u_0 \in S, g \in \G$, and for every $u_0 \in S$,
\begin{equation*}
\lim_{|g| \to 0} \eps_0(u_0,g) = 0.
\end{equation*}
\end{definition}
\begin{definition}\label{rACdef}
We say that $P_t$ has \emph{asymptotic coupling} property \emph{restricted to the action of $\G$} on a set $S \subseteq X$, in short $\mathrm{(rAC)_S}$, 
if there exists a sequence of times $t_n \ge 0$ and a sequence of positive real numbers $\delta_n \to 0$, such that for every $u_0 \in S$, there exist $r=r(u_0) > 0$ so that for every $g \in B_{r(u_0)}$ and for every $\ph \in \mathscr L^\infty(X)$ with $\|\ph\|_{\G-\Lip} < \infty$, we have 
\begin{equation}\tag{rAC} \label{rAC}
\big| P_{t_n} \ph(u_0) - P_{t_n} \ph(\tau_g(u_0)) \big| \le 2\eps_0(u_0,g) \|\ph\|_{\infty} + \delta_n C(u_0,g) \|\ph\|_{\G-\Lip},
\end{equation}
where $\eps_0(u_0,g) < 1$ and $C(u_0,g) < \infty$ for every $u_0 \in S$ and $g \in B_{r(u_0)}$.
\end{definition}
\begin{remark}
By definition, we have that the property $\mathrm{(rASF)_S}$ implies the the $\mathrm{(rAC)_S}$ property with an arbitrarily small constant $\eps_0(u_0,g)$.
\end{remark}
\begin{definition}
We say that a Borel probability measure $\mu$ on $X$ is \emph{invariant} for $P_t$ if for every $\ph \in \mathscr L^\infty(X)$ and for every $t \ge 0$, we have 
\begin{equation}\label{invariance}
\int_X P_t\ph(u) d \mu(u) = \int_X \ph(u) d \mu(u).
\end{equation}
\end{definition}
\subsection{Support theorems} We are now able to state our main support theorem, which should be seen as a generalisation of Proposition \ref{suppclassic} to our setting.
\begin{theorem}[Support theorem \rm{I}] \label{supp1}
Suppose that $P_t$ satisfies $\mathrm{(rAC)_X}$, and let $\mu, \nu$ be \emph{Radon} probability measures which are invariant for $P_t$ and such that $\mu\perp\nu$. 
Then there exist two disjoint sets  $U_0, U_1 \subseteq X$, both open with respect to the topology induced by $d_\G$ and measurable with respect to the Borel sigma-algebra on $X$, such that 
$$\mu(U_0) = 1, \nu(U_0) = 0, \hspace{5pt} \text{and } \mu(U_1) = 0, \nu(U_1) = 1. $$
Moreover, 
\begin{enumerate}
\item if $\eps_0(u_0,g) < \frac 12$ for every $u_0 \in X, |g| \le r(u_0)$, then the closures of $U_0, U_1$ in the topology induced by $d_\G$ are also disjoint. In particular, this holds if $P_t$ satisfies $\mathrm{(rASF)_X}$.
\item if $r(u_0) = \infty$ for every $u_0 \in X$, we can choose $U_0$ to be $\G$-invariant, i.e.\ $\tau(\G \times U_0) = U_0$, and $U_1 = (U_0)^c$.
\end{enumerate}
\end{theorem}
In the particular setting of this paper, it is convenient to repackage (ii) of this support theorem in the following statement, which will have a direct application in the proof of Theorem \ref{uniqueness2d}.
\begin{theorem}[Support theorem \rm{II}] \label{supp2}
Let $X/\G$ be the set of the orbits for the action of $\G$, i.e. 
$$ X/\G = \{\tau(\G \times \{x\}) : x \in X\}, $$
and let $\pi: X \to X/\G$ be the canonical projection. Let $\pi_\# \Bo$ be the projected sigma algebra over $X/\G$, i.e. 
$$ E \in \pi_\# \Bo \Longleftrightarrow \pi^{-1}(E) \in \Bo(X). $$
Suppose that the semigroup $P_t$ satisfies $\mathrm{(rAC)_X}$ with $r(u_0) = \infty$ for every $u_0 \in \G$. Then two Radon probability measures $\mu, \nu$ which are invariant with respect to $P_t$ satisfy $$\mu \perp \nu\iff \pi_\# \mu \perp \pi_\# \nu,$$ 
and similarly 
$$\mu \ll \nu \iff \pi_\# \mu \ll \pi_\# \nu.$$
\end{theorem}
\begin{example}\label{ex:ac->rac}
Let $X$ be a Polish space, and let  $P_t$ be a Markov semigroup on $\mathscr L^\infty(X)$ such that for an appropriate sequence of times $t_n$, for every $x \in X$ and every $\eta > 0$, there exists a radius $r = r(x,\eta) > 0$  and a sequence $\delta_n = \delta_n(x) \to 0$ such that 
\begin{equation}
\sup_{y \in B(x,r)} |P_{t_n} \ph(x) - P_{t_n} \ph(y)| \le  \eta \|\ph\|_{\infty} + \delta_n(x) \| \ph \|_{\Lip}. \label{asf}
\end{equation}
It is well known that this condition, together with the assumption that $P_t$ is Feller, implies the asymptotic strong Feller property for $P_t$ (see for instance \cite[Proposition 3.12]{hmasf}). 
As we saw in the introduction, the main consequence of the asymptotic strong Feller property is that for every two invariant measures $\mu, \nu$ with $\mu \perp \nu$, we have that $\supp(\mu) \cap \supp(\nu) = \emptyset$ (as in Proposition \ref{suppclassic}). We can derive this result as a consequence of Theorem \ref{supp1} as well. We consider 
$$\G = S^\infty(X):=\{ f: X \to X \text{ bijective }: \sup_{x \in X} d(x, f(x))\}, $$
with 
$$ |f| := \|f\|_{\infty} = \sup_{x \in X} d(x, f(x)), $$
and we define the action of $\G$ over $X$ simply by 
$$ \tau_f(x) := f(x). $$
It is easy to check that Assumptions \ref{ass:metric}, \ref{ass:markov}, \ref{ass:group}, \ref{ass:action} hold. Moving to Assumption \ref{ass:tautopology}, we first notice that in general $\tau_f$ is not continuous.\footnote{When $X$ is uncountably infinite, we can build $f$ so that $\tau_f$ is not even \emph{measurable}. For instance, for $X = [0,1] \subset \R$, we can take a non-measurable subset $E \subset [0,1]$ with the cardinality of $\R$, and consider a bijective map $f$ so that $f(E) = [0,\frac12)$ and $f(E^c) = [\frac 12,1]$. While it is not necessary here, it is fairly easy to avoid this measurability issue, simply by adding to the definition of $S^\infty(X)$ the requirement $f$, $f^{-1}$ measurable .}  
Indeed, for every $x_0, y_0 \in X$, one can consider the transposition
\begin{equation*}
f_{(x_0,y_0)}(x) := \begin{cases}
y_0 & \text{ if } x=x_0, \\
x_0 & \text{ if } x=y_0, \\
x & \text{ otherwise.}
\end{cases}
\end{equation*}
We have that $f_{(x_0,y_0)} \in S^\infty(X)$, with $|f_{(x_0,y_0)}| = d(x_0,y_0)$. However, unless $x_0, y_0$ are both isolated points, it is easy to check that $f_{(x_0,y_0)}$ is not a continuous map on $X$. Nevertheless, Assumption \ref{ass:tautopology} still holds. First of all, we have that 
\begin{equation*}
d(x,\tau_f(x)) \le \|f\|_{\infty} = |f|,
\end{equation*}
so we have \eqref{equicont}.
Moreover, if $K$ is compact, it is easy to check that  
\begin{equation}
\tau(B_r \times K) = \{y \in X: \exists x \in K \text{ s.t. } d(x,y) \le r\} = \{ y \in X: d(y,K) \le r\} ,  
\end{equation}
which is a closed subset of $X$, hence measurable. Note that in order to show the first equality, we need to use the fact that for every $x,y$ with $d(x,y) \le r$, there exist an element $f \in S^\infty$ with $|f| \le r$ such that $\tau_f(x) = y$. 
It is not hard to find a Polish space $X$ and two points $x,y$ so that no homeomorphism satisfies this property.\footnote{For instance $X = [0,1]$, $x=0$, $0 < y< 1$.} This is why it is convenient not to require continuity of the action $\tau$ on $X$ in Assumption \ref{ass:tautopology}.

We have that for $x,y \in X$, 
\begin{align*}
d_\G(x,y) &=   \inf \{ |f|: f \in S^\infty(X), x = \tau_f(y) \} \\
&=  \inf \{ \|f\|_\infty: f \in S^\infty(X), x = f(y) \} \\
&=  d(x,y).
\end{align*}
Notice that the last equality is achieved by taking $f = f_{(x,y)}$. As a consequence, we have that 
$$ \| \ph \|_{\G-\Lip} = \sup_{x,y \in X, x\neq y} \frac{|\ph(x) - \ph(y)|}{d(x,y)} = \| \ph \|_{\Lip}. $$
Therefore, by \eqref{asf}, we have that $P_t$ satisfies the $\mathrm{(rASF)_X}$ property,
so we can apply the result of Theorem \ref{supp1}. Since $d_\G = d$, we have that if $\mu, \nu$ with $\mu \perp \nu$ are invariant measures for $P_t$,\footnote{Recall that every Borel probability measure on a Polish space is a Radon measure.} then there exist two open sets $U_0, U_1 \subset X$ with disjoint closures such that 
$$ \mu(U_0) = 1, \nu(U_0) = 0, \hspace{3pt} \mu(U_1) = 0, \nu(U_1) = 1. $$
Therefore, 
$$ \supp(\mu) \subseteq \overline{U_0}, \supp(\nu) \subseteq \overline{U_1} \Rightarrow \supp(\mu) \cap \supp(\nu) = \emptyset. $$
\end{example} 
\begin{remark} \label{rk:weakness}
Proceeding as in Example \ref{ex:ac->rac}, when $X$ is a Polish space and $\G = S^\infty(X)$, we can relate the various assumptions of this section with the various results that are exist in the literature, see for instance  \cite{hmasf,GHMR17,BKS,bs,BW}.
In particular, one can observe that the definitions of the $\mathrm{(rASF)_S}$ and  $\mathrm{(rAC)_S}$ properties are strictly more restrictive than what the existing theory for asymptotic strong Feller and asymptotic coupling properties (respectively) allows. This is due to the following two requirements in Definition \ref{rASFdef} and \ref{rACdef}, which are both avoidable in the classical case.
\begin{itemize}
\item The sequence of times $t_n$ is taken to be the same for every $u_0 \in S$, 
\item The sequence $\dl_n$ is not allowed to depend on $u_0 \in S$. 
\end{itemize}
In this paper, we ask for these restriction due to the extremely weak properties of the topology induced by the distance $d_\G$. 
Namely, the main obstacles to removing the requirements above are the lack of separability of $(X, d_\G)$ and the fact that the measure considered will not be (in general) Radon measures on the space $(X, d_\G)$. 
This means that, in order to develop the theory, some extra uniformity in the base point $u_0$ is required, since it cannot be recovered via $\sigma$-additivity.
\end{remark}

\subsection{Proof of the support theorems}
We start with a couple of preparatory lemmas.
\begin{lemma} \label{lem:bigcap}
Let $K \subseteq X$ be a compact set. Then 
\begin{equation*}
K = \bigcap_{\delta > 0} \tau( B_{\delta} \times K).
\end{equation*}
\end{lemma}
\begin{proof}
Since $e \in B_\delta$ for every $\delta > 0$ and $\tau_e(K)=K$, we have that $K \subseteq \bigcap_{\delta > 0} \tau( B_{\delta} \times K)$. Therefore, we just need to show the reverse inclusion $\supseteq$. Let $u \in \bigcap_{\delta > 0} \tau( B_{\delta} \times K)$. 
Then we have that for every $\delta > 0$, there exist $x_{\delta} \in K$ and $g_\delta \in B_\delta$ so that 
$$ x= \tau_{g_\delta}(x_\delta). $$
Since $K$ is compact, there exists $\delta_n \to 0$ so that $x_{\delta_n}$ has a limit in $K$. Let $x_0 := \lim_{n \to \infty} x_{\delta_n}$. Since $x_0 \in K$, we just need to show that $x = x_0$. By \eqref{equicont}, we have that 
\begin{align*}
d(x,x_0) = \lim_{n \to \infty} d(x, x_{\delta_n}) = \lim_{n \to \infty} d(\tau_{g_{\delta_n}}(x_{\delta_n}), x_{\delta_n}) \le \lim_{n \to \infty} \sup_{x \in K, |g| \le \delta_n} d(x,\tau_g(x)) = 0,
\end{align*}
so $x = x_0 \in K$.

\end{proof}
\begin{lemma} \label{dglip}
Let $A \subseteq X$, and define 
\begin{equation}
d_\G (x,A) :=  \min\big( \inf \{ |g| : x \in \tau_g(A) \}, 1).
\end{equation}
Then the function $d_\G(\cdot, A): X \to \R$ is $\G$-Lipschitz with $\| d_\G(\cdot, A) \|_{\G-\Lip} \le 1$. 
\end{lemma}
\begin{proof}
Let $x, y \in X$. If $x,y \not \in \tau(\G \times A)$, then $d_\G(x,A) = d_\G(y,A) = 1$, and so $|d_\G(x,A) - d_\G(y,A)| = 0$. If $x \not \in \tau(\G \times A)$ and $y \in \tau(\G \times A)$, then $x \not \in \tau(\G \times \{y\})$ either, and so $d_\G(x,y) = \infty$. 
Therefore, we just need to prove that 
$$ d_\G(x,A) - d_\G(y,A) \le d_\G(x,y) $$
under the assumption that $x,y \in \tau(\G \times A)$. Fix $\eps > 0$, and let $g,h \in \G$ be such that $\tau_{g^{-1}}(y) \in A$, $\tau_{h}(y) = x$, and 
\begin{align*}
d_\G(y,A) \ge \min( |g|, 1) - \eps , \hspace{5pt} d_\G(x,y) \le |h| + \eps.
\end{align*}
Then we have that 
$$\tau_{g^{-1}h^{-1}}(x) = \tau_{g^{-1}}(\tau_{h^{-1}}(x)) = \tau_{g^{-1}}(y) \in A,$$ 
and so
\begin{align*}
d_\G(x,A) \le \min( |hg|, 1) \le |h| + \min( |g|,1) \le d_\G(x,y) + d_\G(y,A) +  2 \eps.
\end{align*}
We conclude by sending $\eps \to 0$.
\end{proof}

We focus on the proof of Theorem \ref{supp1}, and Theorem \ref{supp2} will be a proven at the end of this section as a corollary. We then take $\mu, \nu$, and assume that $\mu \perp \nu$. Therefore, there exists a Borel set $E \subset X$ such that 
$$ \mu(E) = 1, \nu(E) = 0. $$
Since $\mu$ is a Radon measure, for every $\eta > 0$ there exists $E_\eta \subseteq E$ compact such that $\mu(E_\eta) \ge 1- \eta$. 
Moreover, since $\nu(E_\eta) \le \nu(E) = 0$, by Lemma \ref{lem:bigcap} and \eqref{measurability}, there exists some $1 > \eta' = \eta'(\eta) > 0$ such that $\nu(\tau(B_{\eta'} \times E_\eta)) \le \eta$. For $\eta > 0$, define the function
\begin{equation}
\psi_\eta(u):= \min\Big( \frac 1 {\eta'} d_\G(u,E_\eta), 1\Big).
\end{equation}
From Lemma \ref{dglip}, it follows that 
\begin{equation}\label{psietalip}
\|\psi_{\eta}\|_{\G-\Lip} \le \frac 1 {\eta'}.
\end{equation}
Moreover, using Assumption \ref{ass:tautopology}, we can check that $\psi_\eta$ is a measurable function. Indeed we have that for $s \in \R$, 
\begin{equation*}
\{ \psi_\eta \le s \} = 
\begin{cases}
X & \text{ if } s \ge 1, \\
\tau(B_{\eta's} \times E_\eta) & \text{ if } 0 \le s < 1, \\
\emptyset & \text{ if } s < 0,
\end{cases}
\end{equation*}
and every one of these sets is measurable by \eqref{measurability}.

Since $P_t$ satisfies $\mathrm{(rAC)_X}$, for every $\eta >0$, we can find an index $n_\eta \in \N$ such that 
\begin{equation} \label{deltansmall}
\frac{\delta_{n_\eta}}{\eta'(\eta)} \le \eta.
\end{equation}
\begin{lemma} \label{lem:01limit}
We have that 
\begin{gather}
\lim_{\eta \to 0} \| P_{t_{n_\eta}} \psi_{\eta} \|_{L^1(\mu)} = 0, \label{0limit}\\
\lim_{\eta \to 0} \| \1 - P_{t_{n_\eta}} \psi_{\eta} \|_{L^1(\nu)} = 0 \label{1limit}.
\end{gather}
\end{lemma}
\begin{proof}
Since $P_t$ is a Markov semigroup and $\psi_\eta \ge 0$, we have that $P_{t_{n_\eta}} \psi_\eta \ge 0$ as well. Therefore, by \eqref{invariance},
\begin{align*}
\| P_{t_{n_\eta}} \psi_{\eta} \|_{L^1(\mu)} & = \int P_{t_{n_\eta}} \psi_{\eta} (u) d \mu (u) \\
&= \int \psi_{\eta} (u) d \mu (u) \\
&\le \int \1_{E_\eta^c}(u) d \mu(u) \\
&\le \eta,
\end{align*}
which is converging to $0$ as $\eta \to 0$. This shows \eqref{0limit}. Similarly, since $\psi_\eta \le \1$, and $P_t$ is Markov, we have that 
$\1 - P_{t_{n_\eta}} \psi_\eta \ge 0$. Therefore, again by \eqref{invariance},
\begin{align*}
\| \1 - P_{t_{n_\eta}} \psi_{\eta} \|_{L^1(\nu)} & = 1 -  \int P_{t_{n_\eta}} \psi_{\eta} (u) d \nu (u) \\
&= \int (1-\psi_{\eta}) (u) d \nu (u) \\
& \le \int \1_{\tau(B_{\eta'} \times E_\eta)}(u) d \nu (u) \\
& \le \eta,
\end{align*}
and this shows \eqref{1limit}.
\end{proof}
In view of Lemma \ref{lem:01limit}, recalling that convergence in $L^1$ implies a.e.\ convergence on a subsequence, we have that on a sequence $\eta_k \to 0$, 
\begin{align}
P_{t_{n_{\eta_k}}} \psi_{\eta_k} (u) \to 0 \text{ for }\mu-\text{a.e.\ }u, \hspace{5pt} P_{t_{n_{\eta_k}}} \psi_{\eta_k} (u) \to 1 \text{ for }\nu-\text{a.e.\ }u. \label{aeconv}
\end{align}
For convenience (and abusing slightly of notation), we relabel $t_k := t_{n_{\eta_k}}$, $\delta_k := \delta_{n_{\eta_k}}$. We define the sets
\begin{gather}
S_0 := \{ u \in X: \lim_{k \to \infty} P_{t_k}\psi_{\eta_k}(u) = 0\}, \label{S0def}\\
S_1 := \{ u \in X: \lim_{k \to \infty} P_{t_k}\psi_{\eta_k}(u) = 1\}. \label{S1def}
\end{gather}
Since $P_{t_k}\psi_{\eta_k}$ are measurable functions for every $k$, we have that the sets $S_0$ and $S_1$ are both measurable. Moreover, in view of \eqref{aeconv}, we have that 
\begin{equation}\label{S01separation}
\mu(S_0) = 1, \hspace{10pt} \nu(S_1) =1,
\end{equation}
and clearly $S_0 \cap S_1 = \emptyset$. 
\begin{lemma}\label{S0invariance}
Let $u_0 \in X$, and let $g \in B_{r(u_0)}$, where $r(u_0)$ is as in \eqref{rAC}. Then 
\begin{equation}\label{taucomparison}
\limsup_{k \to \infty} \big| P_{t_k}\psi_{\eta_k}(\tau_g(u_0)) - P_{t_k}\psi_{\eta_k}(u_0) \big| \le \eps_0(u_0,g).
\end{equation}
In particular, if $u_0 \in S_0$, then $\tau_g(u_0) \not \in S_1$, and similarly, if $u_0 \in S_1$, then $\tau_g(u_0) \not \in S_0$.
\end{lemma}
\begin{proof}
First of all, we notice that for every $\eta \ge 0$, 
$$ \| \psi_\eta - \frac 12 \|_{\infty} \le \frac 12. $$
Therefore, by \eqref{rAC}, \eqref{psietalip}, and \eqref{deltansmall}, we have that 
\begin{align*}
&\big| P_{t_k}\psi_{\eta_k}(\tau_g(u_0)) - P_{t_k}\psi_{\eta_k}(u_0) \big| \\
&= \big| P_{t_k}\big(\psi_{\eta_k}-\frac12\big)(\tau_g(u_0)) - P_{t_k}\big(\psi_{\eta_k}-\frac12\big)(u_0) \big|\\
&\le 2\eps_0(u_0,g) \big\| \psi_{\eta_k}-\frac12 \big\|_{\infty} + \delta_k C(u_0,g) \|\psi_{\eta_k}\|_{\G-\Lip}\\
&\le \eps_0(u_0,g) + C(u_0,g) \frac{\delta_k}{\eta'(\eta_k)} \\
&\le \eps_0(u_0,g) + C(u_0,g)\eta_k.
\end{align*}
Taking the $\limsup$ as $k \to \infty$, we obtain \eqref{taucomparison}. 

Therefore, if  $u_0 \in S_0$, from \eqref{S0def} we have that 
\begin{align*}
\limsup_{k \to \infty} P_{t_k}\psi_{\eta_k}(\tau_g(u_0)) &= \limsup_{k \to \infty} (P_{t_k}\psi_{\eta_k}(\tau_g(u_0)) - P_{t_k}\psi_{\eta_k}(u_0)) \le  \eps_0(u_0,g)  < 1,
\end{align*}
so in particular $\tau_g(u_0) \not \in S_1$. Similarly, if $u_0 \in S_1$, from \eqref{S1def} we obtain 
\begin{align*}
\liminf_{k \to \infty} P_{t_k}\psi_{\eta_k}(\tau_g(u_0)) &= 1 + \liminf_{k \to \infty} (P_{t_k}\psi_{\eta_k}(\tau_g(u_0)) - P_{t_k}\psi_{\eta_k}(u_0)) \ge 1 - \eps_0(u_0,g)  > 0,
\end{align*}
so in particular $\tau_g(u_0) \not \in S_0$.
\end{proof}
We are finally able to complete the proof of Theorem \ref{supp1} and Theorem \ref{supp2}.
\begin{proof}[Proof of Theorem \ref{supp1}]
It would be natural to define $U_0$ to be the set 
$$ \{\tau_g(u_0): u_0 \in S_0, g < |B_{r(u_0)/4}| \}, $$
and $U_1$ analogously. The problem with this definition is that there is no guarantee that this set is measurable, hence more work is required.

First of all, since $\mu$ and $\nu$ are Radon measures, we notice that there exist $\sigma$-compact sets ${\wt{S}_0} \subset S_0$ and ${\wt{S}_1}\subset S_1$ respectively such that 
$$ \mu({\wt{S}_0}) = 1, \hspace{5pt} \nu({\wt{S}_1}) = 1. $$ 
We define a function $r_0: X \to \R$ by 
\begin{equation}
r_0(u_0) := \sup \{r \ge 0: \tau(B_r \times \{u_0\}) \cap {\wt{S}_1} = \emptyset \}
\end{equation}
if such an $r \ge 0$ exists, and $r_0(u_0) = 0$ otherwise. 
By Assumptions \ref{ass:group} and \ref{ass:action}, we also have that 
\begin{equation}
r_0(u_0) = \sup \{r \ge 0: u_0 \not \in \tau(B_r \times {\wt{S}_1}) \},
\end{equation}
when $r_0(u_0) > 0$, and 
\begin{equation}\label{S1tdef}
r_0(u_0) = 0 \Leftrightarrow u_0 \in \bigcap_{r > 0} \tau(B_r \times {\wt{S}_1}) = \bigcap_{n \in \N} \tau(B_{n^{-1}} \times {\wt{S}_1}) =: \overline{{\wt{S}_1}},
\end{equation}
where $\overline{{\wt{S}_1}}$ corresponds exactly to the closure of the set ${\wt{S}_1}$ in the topology induced by $d_\G$.
Therefore, by Assumption \ref{ass:tautopology}, for every $r > 0$, the set
\begin{align*}
\{ r_0 \ge r \} = \bigcap_{r' < r} \tau(B_{r'} \times {\wt{S}_1})^c
\end{align*}
is measurable, and clearly $\{ r_0 \ge 0\} = X$, so $r_0$ is a measurable function. Moreover, by Lemma \ref{S0invariance}, we have that for every $u_0 \in S_0$, $r_0(u_0) \ge r(u_0)$. Finally, we can check that $\| r_0 \|_{\G-\Lip} \le 1$. Indeed, proceeding as in the proof of Lemma \ref{dglip}, for $x,y \in X$, we can assume that $d_\G(x,y) < \infty$ and $x,y \not \in \overline{{\wt{S}_1}}$, and pick $g \in \G$ so that $x =\tau_g(y)$, 
\begin{align*}
d_\G(x,y) \ge |g| - \eps.
\end{align*}
We obtain that 
\begin{align*}
r_0(x) &= \sup \{r \ge 0: x \not \in \tau(B_r \times {\wt{S}_1})  \} \\
&= \sup \{r \ge 0: \tau_g(y) \not \in \tau(B_r \times {\wt{S}_1}) \} \\
&= \sup \{r \ge 0: y \not \in \tau(\tau_{g^{-1}}(B_r) \times {\wt{S}_1})  \}\\
&\ge \sup \{r \ge 0: y \not \in \tau(B_{r+|g|} \times {\wt{S}_1}) \}\\
&= r_0(y) - |g| \\
&\ge r_0(y) - d_\G(x,y) + \eps.
\end{align*}
Therefore, by taking $\eps \to 0$, we obtain that for every $x,y \in X$, 
$$ r_0(y) - r_0(x) \le d_\G(x,y). $$
By swapping the roles of $x$ and $y$, we deduce that $\| r_0 \|_{\G-\Lip} \le 1$. 

Proceeding similarly, we can define the map $r_1: X \to \R$ by 
\begin{equation}
r_1(u_0) := \sup \{r \ge 0: \tau(B_r \times \{u_0\}) \cap {\wt{S}_0} = \emptyset \},
\end{equation}
and $r_1(u_0) = 0$ if $u_0 \in \overline{{\wt{S}_0}}$ (defined analogously to $\overline{{\wt{S}_1}}$ in \eqref{S1tdef}). This map will satisfy the same properties as $r_0$, that is, $r_1$ is measurable, $\| r_1\|_{\G-\Lip} \le 1$, and $r_1(u_0) \ge r(u_0)$ for every $u_0 \in S_1$. We are finally ready to define the sets $U_0$ and $U_1$. Let 
\begin{equation*}
U_0 := \{ r_0 > 2r_1 \}, \hspace{5pt} U_1 := \{ r_1 > 2 r_0 \}.
\end{equation*}
Since $r_0, r_1$ are $\G$-Lipschitz, the sets $U_0$ and $U_1$ are open with respect to the topology induced by $d_\G$. Moreover, 
$$ U_0 \cap U_1 \subseteq \{r_0 < 0 \} \cap \{ r_1 < 0 \} = \emptyset. $$
For $u_0 \in \wt S_0$, by Lemma \ref{S0invariance} have that 
\begin{align*}
r_0(u_0) \ge r(u_0) > 0, r_1(u_0) = 0,
\end{align*}
so $\wt S_0 \subseteq U_0$, and similarly $\wt S_1 \subseteq U_1$. Therefore, 
$$ \mu(U_0) = 1, \nu(U_1) = 1. $$
If moreover $\eps_0(u_0,g) < \frac 12$ for every $u_0 \in X$, $g \in B_{r(u_0)}$, we consider 
$$ \overline{U_0} \cap \overline{U_1} \subseteq \{r_0 = 0 \} \cap \{ r_1 = 0 \} = \overline{{\wt{S}_0}} \cap \overline{{\wt{S}_1}}. $$
We just need to show that in this case $\overline{{\wt{S}_0}} \cap \overline{{\wt{S}_1}}$ is empty. By definition, 
$$\overline{{\wt{S}_0}} \cap \overline{{\wt{S}_1}} \subseteq \{ u_0: \tau(B_{r(u_0)} \times \{u_0\}) \cap {\wt{S}_0} \neq \emptyset, \tau(B_{r(u_0)} \times \{u_0\}) \cap {\wt{S}_1} \neq \emptyset\}. $$
Suppose by contradiction that this set is not empty, and let $u_0 \in \overline{{\wt{S}_0}} \cap \overline{{\wt{S}_1}}$. Then there exist $g_0, g_1 \in B_{r(u_0)}$ so that $\tau_{g_0}(u_0) \in S_0$, $\tau_{g_1}(u_0) \in S_1$. Therefore, by \eqref{S0def}, \eqref{S1def} and \eqref{taucomparison},
\begin{align*}
1 &= \lim_{k \to \infty} P_{t_k}\psi_{\eta_k}(\tau_{g_1}(u_0)) - P_{t_k}\psi_{\eta_k}(\tau_{g_0}(u_0)) \\
&\le \limsup_{k \to \infty} \big|P_{t_k}\psi_{\eta_k}(\tau_{g_1}(u_0)) - P_{t_k}\psi_{\eta_k}(u_0) \big| + \limsup_{k \to \infty} \big| P_{t_k}\psi_{\eta_k}(u_0) - P_{t_k}\psi_{\eta_k}(\tau_{g_0}(u_0))\big|\\
&\le \eps(u_0,g_1) + \eps(u_0,g_0) \\
&< 1,
\end{align*}
contradiction. 

We now move to the case where $r(u_0) = \infty$ for every $U_0$. In this case, we can simply define 
\begin{equation*}
U_0 := \tau(\G \times {\wt{S}_0}),
\end{equation*}
and the fact that $U_0 \cap S_1 = \emptyset$, hence $\nu(U_0) = 0$, follows from Lemma \ref{S0invariance}.
\end{proof}
\begin{proof}[Proof of Theorem \ref{supp2}]
First of all, we notice that if $\pi_\#\mu \perp \pi_\#\nu$, then there exists a set $E \in \pi_\#\Bo$ so that $\pi_\#\mu(E) = 1$, $\pi_\#\nu(E) = 0$, so $\mu(\pi^{-1}(E)) = 1$ and $\nu(\pi^{-1}(E)) = 0$, and we obtain that $\mu \perp \nu$. For the reverse implication, by Theorem \ref{supp1}, there exist a set $U_0 = \tau(\G \times U_0)$ such that $\mu(U_0) = 1$, $\nu(U_0) = 0$. Since 
\begin{align*}
 \pi^{-1}(\pi(U_0)) &= \{ \tau_{g}(u) : g \in \G, \pi(u) = \pi(u_0) \text{ for some }u_0 \in U_0 \} \\
 &= \{ \tau_{g}(u): g \in \G, u = h u_0 \text{ for some } h \in \G, u_0 \in U_0\} \\
 &= \tau( \G \times \tau(\G \times U_0)) \\
 & = \tau(\G \times U_0)\\
 & = U_0 \in \Bo(X),
\end{align*}
we have that $\pi(U_0) \in \pi_\#\Bo$. Therefore, we have that 
\begin{equation*}
\pi_\#\mu(\pi(U_0)) = \mu(\pi^{-1}(\pi(U_0))) = \mu(U_0) = 1, \hspace{5pt} \pi_\#\nu(\pi(U_0)) = \nu(\pi^{-1}(\pi(U_0))) = \nu(U_0) = 0,
\end{equation*}
so $\pi_\#\mu \perp \pi_\#\nu$.

We now move to the other equivalence. The implication $\mu \ll \nu \Rightarrow \pi_\#\mu \ll \pi_\#\nu$ follows from the general property 
that the push-forward of measures preserves absolute continuity. Therefore, we focus on the reverse implication. 
Given $\mu, \nu$ with $\pi_\#\mu \ll \pi_\#\nu$, suppose by contradiction that $\mu \not \ll \nu$. Let then $\mu_s$ be the singular part of $\mu$ with respect to $\nu$, i.e. 
\begin{equation}\label{singulardecomposition}
\mu_s^0 := \frac{d \mu}{d (\mu + \nu)} \1_{\{\frac{d \nu}{d (\mu + \nu)} = 0\}} (\mu + \nu), \hspace{5pt} \mu_s := \frac{\mu_s^0}{\mu_s^0(X)},
\end{equation}
where $\frac{d \mu}{d (\mu + \nu)}, \frac{d \nu}{d (\mu + \nu)}$ denote the Radon-Nykodim derivatives. It is a standard argument to see that $\mu_s$ is a probability measure invariant for $P_t$ (see for instance \cite[Lemma 5.11]{FoToerg}). Moreover, we have that $\mu_s \perp \nu$. Therefore, by the previous part of the proof, $\pi_\# \mu_s \perp \pi_\# \nu$. 
But $\mu_s \ll \mu$, so $\pi_\# \mu_s \ll \pi_\# \mu \ll \pi_\# \nu$, which is a contradiction.
\end{proof}

\section{The hyperbolic $P(\Phi)_2$-model}

\subsection{On of the $P(\Phi)_2$ measure and Wick renormalisation}
Let $\T^2 = \Big(\R / [0,2\pi]\Big)^2$ denote the standard 2-dimensional torus, and let $dx$ be the Lebesgue measure on the torus.
We first start by considering the Gaussian measure $\rho_0$ \eqref{rho0}, formally given by 
$$ d\rho_0(u,u_t) = \frac{1}{Z} \exp\Big(- \frac 12 \int_{\T^2} \big(|u|^2 + |\nabla u|^2\big) dx -\frac 12 \int_{\T^2} |u_t|^2 dx\Big) du du_t. $$
By expressing the norms above in Fourier series, we get that (formally) the measure above corresponds to the measure
 $$ d\rho(u,u_t) = \prod_{n\in \Z^2} \frac{1}{Z^1_n} \exp\Big(-\frac1{8\pi^2} \jb{n}^2 |\ft{u}(n)|^2\Big) d \ft{u}(n)  \prod_{n\in \Z^2} \frac{1}{Z^2_n} \exp\Big(-\frac1{8\pi^2}  |\ft{u_t}(n)|^2\Big) d \ft{u_t}(n)  $$
restricted to the set $\{\ft{u}(-n) = \cj{\ft{u}(n)}, \ft{u_t}(-n) = \cj{\ft{u_t}(n)}\}$, where $\jb{n} :=  \sqrt{1 + |n|^2}$. Therefore, we can write 
\begin{equation} \label{rho0law}
\rho(u,u_t) = \Law(\U), 
\end{equation}
where $\U = (U,V)$ is given by 
\begin{align*}
U &= \frac{1}{2\pi}\Re{\Big(\sum_{n \in \Z^2} \frac{g_n}{\jb n} e^{i n \cdot x}\Big)},\\
V &= \frac{1}{2\pi}\Re{\Big(\sum_{n \in \Z^2} {h_n} e^{i n \cdot x}\Big)},
\end{align*}
and $\{g_n\}_{n\in \N}, \{h_n\}_{n \in \N}$ are i.i.d., complex-valued normal random variables. 

For $\sigma \in \R$, define the Hermite polynomials via their generating function, i.e.\ 
\begin{equation}\label{Hermitedef}
H_n(x,\s^2) := \frac{d^n}{dt^n} \big(e^{tx - \frac12 \s^2 t^2}\big),
\end{equation}
or equivalently, $H_n(x,\s^2)$ are the only functions such that the equality 
\begin{equation}
e^{tx - \frac12 \s^2 t^2} = \sum_{n=0}^\infty \frac{t^n}{n!} H_n(x,\s^2) \label{Hermiteseries}
\end{equation}
holds for every $t \in \R$. For $N\in 2^\mathbb \N$ dyadic, define the (sharp) Fourier projector $\pi_N$ via the equality
\begin{equation} \label{piNdef}
\ft{\pi_N f}(n) = \ft f (n) \1_{\{|n|_\infty \le N\}}, 
\end{equation}
where for $(n_1,n_2) \in \Z^2$, we denote $|(n_1,n_2)|_{\infty} = \max\{ |n_1|, |n_2|\}$. We then define the variance 
\begin{equation}\label{sNdef}
\s_N^2 = \E|{\pi_N U}(x)|^2 = \frac{1}{4\pi} \sum_{|n|_{\infty} \le N} \frac{1}{\jb{n}^2} \sim \log N.
\end{equation}
Finally, for a function $u \in H^{-\eps}$, we denote 
\begin{equation} \label{wickdef}
\wick{u^j} = \lim_{N \to \infty} H(\pi_N u(\cdot), \s_N),
\end{equation}
whenever this limit exists in the Sobolev space $H^{-\eps}$. We have the following properties
\begin{proposition}[Properties of the Wick powers] \label{wickproperties}
Fix $j \in \N$. Then, for every $0 \le \eps \le \eps_j \ll 1$, we have the following.
\begin{enumerate}
\item[\rm 1.] Let $u\in H^{-\eps}$ be such that $\wick{u^0}, \wick{u^1}, \dotsc, \wick{u^j}$ are well-defined and belong to the Sobolev space $W^{-\eps,4}$. Let $v \in H^{1-\eps}$. Then $\wick{(u+v)^j}$ is also well defined, and it satisfies
$$ \wick{(u+v)^j} = \sum_{h \le j} \binom{h}{j} \wick{u^h}v^j \in H^{-\eps}. $$
\item[\rm 2.] Let $p(x) = a_j x^j + \dotsb + a_0$ be a polynomial of degree $j$, and let $\{p_h\}_{1\le h \le j}$ be the unique polynomials such that 
$$ p(x+y) - p(y) = \sum_{h=1}^j p_h(x) y^h. $$
Suppose that $u\in H^{-\eps}$ is such that $\wick{u^0}, \wick{u^1}, \dotsc, \wick{u^j}$ are well-defined in $W^{-\eps,4}$, and let $v \in H^{1-\eps}$. Then 
$$ \wick{p(u+v)} - \wick{p(u)} = \sum_{h=1}^j \wick{p_h(u)} v^h. $$
\item[\rm 3.] For $\rho_0$-a.e.\ $u$, we have that $\wick{u^j}$ is well defined and belongs to the Sobolev space $W^{-\eps,\infty}$. Moreover, for every $p < \infty$, we have that 
$$ \int \|\wick{u^j}\|_{W^{-\eps,\infty}}^p d \rho_0(u,u_t) < \infty.$$
\end{enumerate}
\end{proposition}
\begin{proof}
This is essentially a collection of well-known results, so we keep the proof short. From \eqref{Hermiteseries}, we see that we must have 
$$ \frac{d}{dx} H_n(x,\s^2) = n H_{n-1}(x,\s^2), $$
and so by Taylor series expansion, we obtain 
$$ H_j(x+y, \s^2) = \sum_{h=0}^j \binom{n}{j} H_h(x,\s^2) y^h.  $$
By taking limits and exploiting the continuity of the map $W^{-\eps, 4} \times H^{1-\eps} \to H^{-\eps}$ given by $(u,v) \mapsto uv$ (when $\eps$ is small enough), we obtain 1. Notice that when $p(x) = x^j$, then the statement of 2.\ coincides with the statement of 1. Therefore, 2.\ follows from 1.\ by linearity in the coefficients of $p$. 

Finally, 3.\ can be found (for instance) in \cite[Lemma 2.3]{gkot}.
\end{proof}

We conclude this subsection with the following result, which provides a rigorous definition of the measure $\rho$ in \eqref{rhodef}.

\begin{proposition} \label{proprhodef}
Let $P = a_{2k} x^{2k} + \dotsb + a_0$ be a polynomial of even degree, and consider the functional 
$$ F_P(u) = \exp\Big( - \int_{\T^2} \wick{P(u)} dx \Big).  $$
By Proposition \ref{wickproperties}, the functional $F_P(u)$ is well-defined for $\rho_0$-almost every $u,u_t$ (as a function of the first variable). Moreover, we have that for every $p < \infty$,
$$ F_p \in L^p(\rho). $$
Therefore, we define the measure $\rho$ to be 
\begin{equation*}
\rho(u,u_t) = \frac{F_P(u)}{\int F_P(u) d \rho_0(u,u_t)} \rho_0(u,u_t)
\end{equation*}
\end{proposition}
\begin{proof}
This is an immediate corollary of \cite[Theorem V.7]{Simon}. See also \cite[Lemma 2.3]{gkot}.
\end{proof}

\subsection{Local and global theory for \eqref{SDNLW}}
In this subsection, we quickly recap the existing local and global well-posedness theory for \eqref{SDNLW}, 
\begin{align*}
\dt^2 u   + \dt u  +(1-\Dl)  u 
+
\wick{p(u)}
   = \sqrt{2} \xi.
\end{align*} 
For the purpose of this section, it is convenient to write the equation in vectorial form in the variable $\u = \vec{u}{u_t}$,
\begin{equation} \label{SNLW}
\dt \vec{u}{u_t} = \begin{pmatrix}  0 & 1 \\ -(1 - \Delta) & -1 \end{pmatrix} \vec{u}{u_t} - \vec{0}{\wick{p(u)}} + \vec{0}{\sqrt2\xi}.
\end{equation}
We define the linear propagator for this equation to be
\begin{equation}
\begin{aligned}
S(t) &= \exp\bigg( t \begin{pmatrix}  0 & 1 \\ -(1 - \Delta) & -1 \end{pmatrix}\bigg) \\
&= e^{-\frac t2}
\begin{pmatrix}
\cos\Big(t\sqrt{\frac34- \Delta}\Big) + \frac12\frac{\sin\Big(t\sqrt{\frac34-\Delta}\Big)}{\sqrt{\frac34-\Delta}}
&\frac{\sin\Big(t\sqrt{\frac34-\Delta}\Big)}{\sqrt{\frac34-\Delta}} \\
-\Big(\sqrt{\frac34-\Delta}+\frac1{4\sqrt{\frac34-\Delta}}\Big)\sin\Big(t\sqrt{\frac34-\Delta}\Big) & 
\cos\Big(t\sqrt{\frac34-\Delta}\Big) -\frac12\frac{\sin\Big(t\sqrt{\frac34-\Delta}\Big)}{\sqrt{\frac34-\Delta}}
\end{pmatrix}.
\end{aligned}
\end{equation}
Immediately from this definition, we obtain
\begin{equation}\label{Hsbound}
\|S(t) \u_0\|_{\H^s} \les e^{-\frac t2} \|\u_0\|_{\H^s}.
\end{equation}
for every $s \in \R$ (recall the definition $\H^s:= H^s \times H^{s-1}$). 
As mentioned in Section 1.2, we want to express the solution of \eqref{SNLW} as linear solution + nonlinear remainder. To this scope, we define 
\begin{equation} \label{Stick}
\ve{\psi}[\xi](t):= \int_0^t S(t-t') \vec{0}{\sqrt 2 \xi} d t',
\end{equation}
and we call the components of $\ve\psi[\xi](t)$ 
\begin{equation}\label{stick}
\ve\psi[\xi](t) = \vec{\psi[\xi](t)}{\psi_t[\xi](t)}.
\end{equation}
Notice that this way, $\ve\psi[\xi]$ is the solution of the linear equation
$$
\begin{cases}
\psi_{tt} + \psi_t + (1-\Delta)\psi = \sqrt2 \xi,\\
\ve\psi[\xi](0) = (0,0).
\end{cases}
$$
We point out here that $\ve \psi$ is a low-regularity object. Namely, for every $t > 0$, $\ve\psi \in \H^{-\eps}$ for every $\eps >0$, but $\ve\psi\not\in \H^0$.\footnote{This can easily seen (for instance) from the fact that $\int |\pi_N \psi|^2 - \E\int |\pi_N \psi|^2$ converges almost surely (to $\int \wick{\psi^2}$), but $\E\int |\pi_N \psi|^2 \to \infty$ as $N\to \infty$, which together imply that $\psi \not\in L^2$. Both the almost sure convergence and the divergence of $\E\int |\pi_N \psi|^2$ follow from arguments similar to the ones in Section 3.1.}
Finally, for an initial data $\u_0$, we define $\v$ by $\u(t) := S(t) \u_0 + \ve\psi[\xi](t) + \v(t)$, where $\u$ (formally) solves \eqref{SNLW}. This way, $\v(t) = \vec{v}{v_t}$ will solve the equation
\begin{equation} \label{SNLWv}
\begin{cases}
\dt \vec{v}{v_t} = \begin{pmatrix}  0 & 1 \\ -(1 - \Delta) & -1 \end{pmatrix} \vec{v}{v_t} - \vec{0}{\wick{p(S(t) \u_0 + \psi[\xi](t) + \v(t))}},\vspace{5pt}\\
\vec{v(0)}{v_t(0)} = \vec{0}{0}.
\end{cases}
\end{equation}
We obtain the following local well-posedness statement.
\begin{theorem}[{\cite[Theorem 1.1]{gko}},{\cite[Proposition 4.1]{gkot}}]\label{LWP}
Define 
$$ \Xi(\u_0, \xi)(t) := \big(S(t)\u_0 + \ve\psi[\xi](t), \wick{(S(t)\u_0 + \ve\psi[\xi](t))^2}, \dotsc, \wick{(S(t)\u_0 + \ve\psi[\xi](t))^{2k-1}}\big).$$
For $0<\eps \le \eps_k \ll 1$, let $\u_0 \in \H^{-\eps}$ be such that $\Xi(\u_0, \xi)  \in (L^2_{\loc}(\R^+;W^{-\eps,4}))^{2k-1} \text{ a.s. }$
Then the equation \eqref{SDNLW} is almost surely locally well-posed. More precisely, there exists a random time $T^* = T^*(\|\Xi\|_{(L^2([0,1];W^{-\eps,4}))^{2k-1}}) > 0$ such that the equation \eqref{SNLWv} has a unique solution $\v \in C([0,T^*], \H^{1-\eps})$. For $t < T^*$, we denote 
$$ \Phi_t(\u_0,\xi) := S(t)\u_0 + \ve{\psi}[\xi](t) + \v(t).  $$ 
Moreover, the map $\Xi \mapsto \v$ is continuous as a map from $(L^1([0,1];W^{-\eps,4}))^{2k-1}$ to $C([0,T^*], \H^{1-\eps})$.
\end{theorem}
We remark here that while in the papers cited, Theorem \ref{LWP} was shown only in the case $p(x) = x^{2k-1}$. However, the result for a general polynomial $p$ is a straightforward and easy modification of the arguments in \cite{gkot}. The same remark will hold for Theorem \ref{GWP} below.

On the basis of Theorem \ref{LWP}, we define the set of ``good" initial data to be 
\begin{equation}\label{Gamma2d}
 \Gamma :=  \{\u_0 \in \H^{-\eps}: \Xi(\u_0, \xi)  \in \bigotimes_{j=1}^{2k-1} W^{-\eps, \frac{4(2k-1)}{j}} \text{ a.s.}\}.
\end{equation}
\begin{lemma}\label{2dcompatibility}
Let $\u_0 \in \Gamma$, and suppose that $\v(t) \in C([0,T), \H^{1-\eps})$ solves \eqref{SNLWv} on the interval $[0, T)$. Then 
$\u(t) = S(t) \u_0 + \ve\psi[\xi](t) + \v(t) \in \Gamma$ for every $0 \le t < T$. Moreover, let $t_0 < T$, and suppose that $\v_{t_0} \in C([0,T_1), \H^{1-\eps})$ solves \eqref{SNLWv} on some interval $[0,T_1)$ with initial data $\u(t_0)$ and noise $\tilde \xi:= \xi(t-t_0, \cdot)$. Let 
$$\u_{t_0}(t) = S(t)\u(t_0) + \ve{\psi}[\xi(t-t_0,\cdot)] + \v_{t_0}(t).$$
Then 
\begin{enumerate}
\item For $t_0 \le t \le \min(T, t_0 + T_1)$, 
$$ \u(t) = \u_{t_0}(t-t_0), $$
\item For $t_0 \le t \le t_0 + T_1$, let $\wt{\v}(t) := \u_{t_0}(t-t_0) - S(t)\u_0 - \psi[\xi](t)$. Then $\wt{\v}$ satisfies $\wt{\v} \in C([t_0, t_0+T_1),\H^{1-\eps}),$ and $\wt{\v}$ solves \eqref{SNLWv} on the interval $[t_0, t_0+T_1)$ with initial data $\u_0$ and noise $\xi$.
\end{enumerate}
\end{lemma}
\begin{proof}
For $t\in \R$, define $\xi_t(s) := \xi(t+s)$. Since the law of white noise is independent of time, we have that $\Law(\xi_t) = \Law(\xi)$ for every $t\in \R$. Notice that, by \eqref{Stick} and the semigroup property of $S(t)$, $S(t_1+t_2)= S(t_1)S(t_2)$, we have that $\Xi(S(t) \u_0 + \ve\psi[\xi](t),\xi_t)(s) = \Xi(\u_0,\xi)(t+s).$ Therefore, we have that if $\u_0 \in \Gamma$, then $S(t) \u_0 + \ve\psi[\xi](t) \in \Gamma$ as well. 
The fact that we also have $S(t) \u_0 + \ve\psi[\xi](t) + \v(t) \in \Gamma$ for $\eps\le \eps_k$ small enough is a direct consequence of Proposition \ref{wickproperties}, 1.\ and standard product estimates.

We now move to showing (i). From the uniqueness in Theorem \ref{LWP} (see also {\cite[Proposition 4.1]{gkot}}), it is enough to check that 
$$ \cj{v}(t-t_0):= S(t) \u_0 + \ve\psi[\xi](t) + \v(t) - S(t-t_0)\big(S(t_0) \u_0 + \ve\psi[\xi](t_0) + \v(t_0)\big) - \ve\psi[\xi_{t_0}](t-t_0) $$
solves \eqref{SNLWv} with initial data $\u(t_0)$ and noise $\xi_{t_0}$. This is a straightforward (but tedious) computation. The proof of (ii) is completely analogous.
\end{proof}

The previous lemma allows us to define a stochastic flow on the set 
$$ \wt{X} := \H^{-\eps} \cup \{\infty\}. $$
Here $\infty$ denotes a ``cemetery state", that we use to keep track of when the flow blows up (or it is not well defined).
Indeed, we declare that
\begin{equation} \label{flowdef}
\Phi_t(\u_0, \xi) = 
\begin{cases}
S(t) \u_0 + \ve\psi[\xi](t) + \v(t) &  \begin{aligned}
&\text{if }\u_0 \in \Gamma \text{ and } \v \in C([0,t], \H^{1-\eps}) \text{ solves \eqref{SNLWv},}\\
&\text{if such a solution exists},
\end{aligned}\\
\infty & \text{ otherwise}.  
\end{cases}
\end{equation}
In view of Lemma \ref{2dcompatibility}, the flow as defined will satisfy the semigroup property
\begin{equation}\label{2dsemigroup}
\Phi_{t+s}(\u_0, \xi) = \Phi_{t}(\Phi_{s}(\u_0, \xi), \xi(t-t_0, \cdot)).
\end{equation}
Moreover, we have the following global well-posedeness statement.
\begin{theorem}[Theorem 1.7, \cite{gkot}] \label{GWP}
The renormalised SdNLW~\eqref{SDNLW} is almost surely globally well-posed with initial data distributed according to the renormalised Gibbs measure~$\rho$ in~\eqref{rhodef}. Furthermore, the renormalised Gibbs measure $\rho$ is invariant under the dynamics.

More precisely, for $\rho$-almost every $\u_0$, we have that $\Phi_{t}(\u_0,\xi) \neq \infty$ a.s., and for every bounded, Borel measurable functional $F: \H^{-\eps} \to \R$, 
\begin{equation} \label{invariancedef2d}
\int \E[F(\Phi_{t}(\u_0,\xi))] d \rho(\u_0) = \int F(\u_0) d \rho(u_0). 
\end{equation}
\end{theorem}
\begin{remark}\label{Ptborel}
Since the map $\u_0 \mapsto \Phi_t(\u_0,\xi)$ is not continuous in the topology of $\H^{-\eps}$, one might wonder if it is actually possible to define the semigroup 
$$P_t \ph(\u_0) := \E[\Phi_t(\u_0, \xi)] $$
over the space of bounded, Borel measurable functions over $\wt X$. Namely, it is not a priori clear whether the map 
$$ \u_0 \mapsto \Phi_t(\u_0,\xi) $$
is Borel measurable. However, by \eqref{wickdef}, one obtains that the set 
$$ \Gamma = \{\u_0: \Xi(\u_0, \xi) \in \bigotimes_{j=1}^{2k-1} W^{-\eps, \frac{4(2k-1)}{j}} \text{ a.s.}\} $$
is measurable, since the map $(\u_0,\xi) \mapsto \Xi(\u_0,\xi)$ is measurable on the set where it is well defined, and the set where this map is not well defined can be expressed as the limsup of measurable sets (hence it is measurable). Since furthermore the maps $\Xi \mapsto S(t)\u_0 + \ve{\psi}[\xi](t)$, $\Xi \mapsto \v(t)$ are continuous, we also obtain that the map $ \Gamma \ni \u_0 \mapsto \Phi_t(\u_0,\xi) $ is measurable. 
Finally, one has that $\Phi_t(\u_0,\xi) = \infty$ for $\u_0 \not\in\Gamma$ and $t>0$. We obtain that the map $\wt X \ni\u_0 \mapsto \Phi_t(\u_0,\xi)$ is also measurable. At this point, the semigroup property $P_{t+s} = P_tP_s$ follows from \eqref{2dsemigroup}.
\end{remark}

\subsection{Construction of the Girsanov shift}

Define the set 
\begin{equation}\label{Ydef1}
\begin{aligned}
 Y= \Big\{& \u_0 \in \H^{-\eps}: \Phi_t(\u_0, \xi)\neq \infty \text{ for every } t \ge 0, \\
&\ \text{and for every } p \ge 1, \text{ there exists } C= C(\u_0,p) \text{ s.t. }\\
&\ \E \int_0^{\infty} \sup_{1 \le h \le 2k-1} \Big\|e^{-\frac t{32}}\wick{p_h(\Phi_t(\u_0; \xi))} \Big\|_{W^{-\frac\eps2,4}}^2 dt \le C.  \Big\}
\end{aligned}
\end{equation}
\begin{lemma}\label{lem:Xfull2d}
$Y \in \Bo(\H^{-\eps})$. Moreover, $\rho(Y) = 1$. 
\end{lemma}
\begin{proof}
The fact that $Y$ is Borel measurable follows from the same arguments as in Remark \ref{Ptborel}. In order to show that $\rho(Y) = 1$, it is enough to show that 
\begin{equation*}
\int \Big(\E \int_0^{\infty} \sup_{1 \le h \le 2k-1} \Big\|e^{-\frac t{32}}\wick{p_h(\Phi_t(\u_0; \xi))} \Big\|_{W^{-\frac\eps2,4}}^2 dt\Big) d \rho(\u_0) < \infty. 
\end{equation*}
By Tonelli's theorem, by invariance of $\rho$, and Proposition \ref{wickproperties}, 3.\ together with \ref{proprhodef}, we have that 
\begin{align*}
& \int \Big(\E \int_0^{\infty} \sup_{1 \le h \le 2k-1} \Big\|e^{-\frac t{32}}\wick{p_h(\Phi_t(\u_0; \xi))} \Big\|_{W^{-\frac\eps2,4}}^2 dt\Big) d \rho(\u_0) \\
= & \int_0^\infty \int \E \Big[\sup_{1 \le h \le 2k-1} \Big\|e^{-\frac t{32}}\wick{p_h(\Phi_t(\u_0; \xi))} \Big\|_{W^{-\frac\eps2,4}}^2\Big] d \rho(\u_0) d t \\
= & \int_0^\infty \int \sup_{1 \le h \le 2k-1} \Big\|e^{-\frac t{32}}\wick{p_h(\u_0)} \Big\|_{W^{-\frac\eps2,4}}^2 d \rho(\u_0) dt \\
\les & \int \sup_{1 \le h \le 2k-1} \Big\| \wick{p_h(\u_0)} \Big\|_{W^{-\frac\eps2,4}}^2 d \rho(\u_0)\\
<&\  \infty.
\end{align*}
\end{proof}

The goal of this subsection is to prove the following. 
\begin{proposition}\label{pphi2asfprep}
For every $\u_0 \in Y$ and $\v_0 \in \H^{1-\eps}$, there exist $\eps_0(\u_0,\v_0) > 0$ so that if $F : \H^{-\eps} \to \R$ is a Borel function, Lipschitz with respect to the distance 
$$ d_{\H^{1-\eps}}(\u_0, \u_1):= \min( \| \u_0 - \u_1\|_{\H^{1-\eps}}, 1),$$
we have that 
\begin{equation} \label{pphi2asfestimate}
\begin{aligned}
&|\E [F(\Phi_t(\u_0 + \v_0; \xi)) - F(\Phi_t(\u_0+\v_1; \xi))]| \\
&\le 2\eps_0(\u_0,\v_0, \v_1)\|F\|_{\infty} + e^{-\frac t4}(\norm{\v_0}_{\H^{1-\eps}}+\norm{\v_0}_{\H^{1-\eps}})\|F\|_{\H^{1-\eps}-\Lip},
\end{aligned}
\end{equation}
with $\eps_0(\u_0,\v_0,\v_1) < 1$.

Moreover, in the special case $\v_1 = 0$, we can choose $\eps_0$ so that 
\begin{equation}\label{pphi2cont}
 \lim_{\|\v_0\|_{\H^{1-\eps}} \to 0} {\eps_0(\u_0,\v_0,0)} = 0.
\end{equation}
\end{proposition}
In order to show this, we need two preparatory lemmas.
\begin{lemma}\label{lemma:hdef}
Let $\eps > 0$ small enough, and let $\u_0 \in Y$, $\v_0 \in \H^{1-\eps}$. Then there exists $ h \in L^2(\R_+, \T^2)$, adapted with respect to the natural filtration induced by $\xi$, such that 
\begin{renumerate}
\item For every $t \ge 0$, 
$$ \norm{\Phi_t(\u_0 + \v_0; \xi + h) - \Phi_t(\u_0; \xi)}_{\H^{1-\eps}} \le \norm{\v_0}_{\H^{1-\eps}} e^{-\frac t4},$$
\item There exists an almost surely finite constant $K = K(\u_0,\xi)$ such that 
$$ \int_0^\infty \norm{h(t)}_{L^2(\T^2)}^2 \le K(1+\|\v_0\|_{\H^{1-\eps}}^{2k-2})\|\v_0\|_{\H^{1-\eps}}, $$
and $\E|K| < +\infty$. 
\end{renumerate}
\end{lemma}
\begin{proof}
The statement is trivially true for $\v_0 = 0$, by choosing $h = 0$. 
Therefore, we can safely assume $\v_0 \neq 0$. 
By Proposition \ref{wickproperties}, 2., we have that 
\begin{equation} \label{pexpansion}
\wick{p(\Phi_t(\u_0; \xi) + w)} - \wick{p(\Phi_t(\u_0; \xi))} = \sum_{h=1}^{2k} \wick{p_h(\Phi_t(\u_0; \xi))} w^h,
\end{equation}
where $\{p_h\}_{1 \le h \le 2k-1}$ are appropriate polynomials of degree $h$.
For a given function $\delta(t):\R^+ \to \R^+$ that we will determine later, let $\w=\vec{w}{w_t}$ be the solution of the equation
\begin{equation} \label{weqn}
w_{tt} + w_t + (1-\Delta)w = - \sum_{h=1}^{2k-1} (1 - e^{\dl(t) \Delta}) \wick{p_h(\Phi_t(\u_0; \xi))} w^h
\end{equation}
with initial data $\w_0 = \v_0$.
Therefore, defining $\mathcal T_\delta z := z - e^{\delta \Delta} z$, \eqref{weqn} can be rewritten as 
\begin{equation} \label{weqn2}
w_{tt} + w_t + (1-\Delta)w = - \sum_{h=1}^{2k-1}\mathcal T_{\delta(t)} (\wick{p_h(\Phi_t(\u_0; \xi))}) w^h.
\end{equation}
Since $\norm{\mathcal T_{\delta} z}_{W^{-\eps,4}} \les \delta^{\frac\eps 4}\norm{z}_{W^{-\frac\eps2,4}}$, by choosing 
\begin{equation} \label{deltat}
\delta(t) = \big(A \sup_h \norm{\wick{p_h(\Phi_t(\u_0; \xi))}}_{W^{-\frac\eps2,4}} \norm{\v_0}_{\H^{1-\eps}}^{h-1}\big)^{-\frac 4 \eps}, 
\end{equation}
for $\eps$ small enough and $A$ big enough, we have 
\begin{align*}
\left\|- \sum_{h=1}^{2k-1}\mathcal T_{\delta(t)} (\wick{p_h(\Phi_t(\u_0; \xi))}) w^h\right\|_{H^{-\eps}} \le \frac 1{4} \Big[\sup_{1 \le h \le 2k-1}\Big(\frac{\norm{w}_{H^{1-\eps}}}{\norm{\v_0}_{\H^{1-\eps}}}\Big)^{h-1}\Big] \norm{w}_{H^{1-\eps}}.
\end{align*}
Therefore, by \eqref{Hsbound},
\begin{align*}
&~\norm{\w(t)}_{\H^{1-\eps}} \\
\le &~\norm{S(t) \w(0)}_{\H^{1-\eps}} + \left\| \int_0^t S(t-t') \Big(- \sum_{h=1}^{2k-1}\mathcal T_{\delta(t)} (\wick{p_h(\Phi_t(\u_0; \xi))}) w^h\Big) \d t' \right\|_{\H^{1-\eps}}\\
\le &~e^{-\frac t2} \norm{\w(0)}_{\H^{1-\eps}} + \frac 14 \int_0^t e^{-\frac{t-t'}{2}} \left\|{- \sum_{h=1}^{2k-1}\mathcal T_{\delta(t)} (\wick{p_h(\Phi_t(\u_0; \xi))}) w^h}\right\|_{H^{-\eps}} \d t'\\
\le &\ e^{-\frac t2} \norm{\v_0}_{\H^{1-\eps}} + \frac 14 \int_0^t e^{-\frac{t-t'}{2}} \Big[\sup_{1 \le h \le 2k-1}\Big(\frac{\norm{w(t')}_{H^{1-\eps}}}{\norm{\v_0}_{\H^{1-\eps}}}\Big)^{h-1}\Big] \norm{w(t')}_{H^{1-\eps}} \d t'.
\end{align*}
From this, by an easy Gronwall argument, we obtain 
\begin{equation} \label{wgwth}
\norm{\w}_{H^{1-\eps}} \le e^{-\frac t4} \norm{\v_0}_{\H^{1-\eps}}.
\end{equation}
Moreover, if we define
\begin{equation} \label{hdef}
h(t) := \frac 1 {\sqrt 2} \sum_{j=1}^{2k-1}(1-\mathcal T_{\delta(t)}) (\wick{p_h(\Phi_t(\u_0; \xi))}) w^j,
\end{equation}
by \eqref{pexpansion} it is easy to see that $\tilde \u(t) :=\Phi_t(\u_0;\xi) + \w(t)$ solves the equation \eqref{SNLW} with forcing $\xi + h$, and that $h$ is adapted. Therefore, 
$$\Phi_t(\u_0 + \v_0; \xi + h) =  \Phi_t(\u_0;\xi) + \w(t), $$
so ($\mathrm i$) is proven for this particular choice of $h$. Moreover, for $\eps$ small enough,
\begin{align*}
&\phantom{\les}\norm{h(t)}_{L^2} \\
&\les \sum_{j=1}^{2k-1} \delta^{-\frac \eps 4} \norm{\wick{p_h(\Phi_t(\u_0; \xi))}}_{W^{-\frac \eps 2, 4}} \norm{w}_{H^{1-\eps}}^j \\
&\les A \big(\sup_j \norm{\wick{p_h(\Phi_t(\u_0; \xi))}}_{W^{-\frac\eps2,4}}\big)^2 \big(1 + \norm{\v_0}_{\H^{1-\eps}}^{2k-2}\big) e^{-\frac t 4}\norm{\v_0}_{\H^{1-\eps}}\\
&\le K(t) \big(1 + \norm{\v_0}_{\H^{1-\eps}}^{2k-2}\big) \norm{\v_0}_{\H^{1-\eps}}e^{-\frac t 8},
\end{align*}
for 
$$K(t) \les \big(e^{-\frac t {16}}\sup_h \norm{\wick{p_h(\Phi_t(\u_0; \xi))}}_{W^{-\frac\eps2,4}}\big)^2.$$
Therefore, we have ($\mathrm {ii}$) with 
$$ K \les \int_0^{\infty} \big(e^{-\frac t {16}}\sup_h \norm{\wick{p_h(\Phi_t(\u_0; \xi))}}_{W^{-\frac\eps2,4}}\big)^2 dt $$
Furthermore, by definition \eqref{Ydef1} of $Y$,
\begin{align*}
\E|K| &\les \E\Big|\int_0^\infty e^{-\frac t{16}} \big(e^{-\frac t {32}}\sup_h \norm{\wick{p_h(\Phi_t(\u_0; \xi))}}_{W^{-\frac\eps2,4}}\big)^2 dt \Big| \\
&\les \E \int_0^\infty \big|e^{-\frac t {32}}\sup_h \norm{\wick{p_h(\Phi_t(\u_0; \xi))}}_{W^{-\frac\eps2,4}}\big|^{2} dt \\
&< \infty.
\end{align*}
\end{proof}
\begin{lemma}
Let $0 \le f_1,f_2$ with $\E[f_1], \E[f_2] < \infty$, and let $\eta > 0$. We have that 
\begin{equation}\label{Ediffestimate}
\E[|f_1 - f_2|] \le \E[f_1] + \E[f_2] - \eta\big(\P(\{f_1 \ge \eta\}) + \P(\{f_2 \ge \eta\}) - 1\big).
\end{equation}
\end{lemma}
\begin{proof}
We have that 
\begin{align*}
\E[|f_1 - f_2|] &= \E[(f_1 - f_2) \1_{\{f_1 \ge f_2\}}] + \E[(f_2 - f_1) \1_{\{f_2 \ge f_1\}}] \\
&= \E[(f_1 - f_2) \1_{\{f_1 \ge f_2\ge \eta\}}] + \E[(f_2 - f_1) \1_{\{f_2 \ge f_1\ge \eta\}}] \\
&\phantom{=}+ \E[(f_1 - f_2) \1_{\{f_1 \ge \eta> f_2\}}] + \E[(f_2 - f_1) \1_{\{f_2 \ge \eta > f_1\}}] \\
&\phantom{=}+ \E[(f_1 - f_2) \1_{\{ \eta > f_1 \ge  f_2\}}] + \E[(f_2 - f_1) \1_{\{ \eta > f_2 \ge  f_1\}}] \\
&\le \E[(f_1 - \eta) \1_{\{f_1 \ge f_2\ge \eta\}}] + \E[(f_2 - \eta) \1_{\{f_2 \ge f_1\ge \eta\}}] \\
&\phantom{=}+ \E[f_1 \1_{\{f_1 \ge \eta> f_2\}}] + \E[f_2\1_{\{f_2 \ge \eta > f_1\}}] \\
&\phantom{=}+ \E[f_1 \1_{\{ \eta > f_1 \ge  f_2\}}] + \E[f_2  \1_{\{ \eta > f_2 \ge  f_1\}}] \\
&= \E[f_1 \1_{\{f_1 \ge f_2\}}] + \E[f_2 \1_{\{f_2 \ge f_1\}}] - \eta (\P(\{f_1 \ge f_2 \ge \eta\}) + \P(\{f_2 \ge f_1 \ge \eta\})) \\
&\le \E[f_1] + \E[f_2] - \eta \P(\{\min(f_1,f_2) \ge \eta\}).
\end{align*}
At this point, \eqref{Ediffestimate} follows from 
\begin{align*}
\P(\{\min(f_1,f_2) \ge \eta\}) &= 1 - \P(\{f_1 < \eta\} \cup \{f_2 < \eta\}) \\
&\ge 1 - \P(\{f_1 < \eta\}) - \P(\{f_2 < \eta\})\\
&\ge \P(\{f_1 \ge \eta\}) + \P(\{f_2 \ge \eta\} - 1.
\end{align*}

\end{proof}

\begin{proof}[Proof of Proposition \ref{pphi2asfprep}]
Let $\v_0, \v_1 \in \H^{1-\eps}$, and let $h_{\v_j}$ be as in Lemma \ref{lemma:hdef}. For $M>0$, let $\tau_{\v_j}^M$ be the first time such that $\norm{h_{\v_j}(t)}_{L^2([0,\tau_{\v_j}^M] \times \T^2)} = M$ (with $\tau_{\v_j}^M = +\infty$ if $\norm{h_{L^2_{t,x}}} < M$ for every $t$), and let $h_{\v_j}^M(t):= h(\min(t,\tau^M_{\v_j}))$. Let 
$$\mathcal E(h)(t):= \exp\Big(-\frac12 \int_0^t \norm{h(t')}_{L^2}^2 + \int_0^t \dual{h(t')}{\xi}_{L^2}\Big).$$
We first observe that in order to prove \eqref{pphi2asfestimate}, by exchanging the roles of $\v_0$ and $\v_1$, it is enough to prove the estimate without the absolute value. 
By Girsanov, we have that 
\begin{align*}
&\phantom{\les} \E [F(\Phi_t(\u_0 + \v_0; \xi)) - F(\Phi_t(\u_0+\v_1; \xi))] \\
&= \E [F(\Phi_t(\u_0 + \v_0; \xi+h_{\v_0}^M))\mathcal E(h_{\v_0}^M) - F(\Phi_t(\u_0 + \v_1; \xi+h_{\v_1}^M))\mathcal E(h_{\v_1}^M)]\\
&= \E [F(\Phi_t(\u_0 + \v_0; \xi+h_{\v_0}^M))\mathcal E(h_{\v_0}^M) \1_{\{\tau^M_{\v_0} < t\}}] \label{2di} \tag{I} \\
&\phantom{=} + \E [F(\Phi_t(\u_0 + \v_0; \xi+h_{\v_0}^M))\mathcal E(h_{\v_0}^M) \1_{\{\tau^M_{\v_0} \ge t\}} - F(\Phi_t(\u_0; \xi))\mathcal E(h_{\v_0}^M) \1_{\{\tau^M_{\v_0} \ge t\}}] \label{2dii} \tag{II}\\
&\phantom{=} + \E [F(\Phi_t(\u_0; \xi)) \mathcal E(h_{\v_0}^M) \1_{\{\tau^M_{\v_0} \ge t\}} - F(\Phi_t(\u_0; \xi))\mathcal E(h_{\v_1}^M) \1_{\{\tau^M_{\v_1} \ge t\}}] \label{2diii} \tag{III}\\
&\phantom{=} - \E [F(\Phi_t(\u_0 + \v_1; \xi+h_{\v_1}^M))\mathcal E(h_{\v_1}^M) \1_{\{\tau^M_{\v_1} \ge t\}}  - F(\Phi_t(\u_0; \xi))\mathcal E(h_{\v_1}^M) \1_{\{\tau^M_{\v_1} \ge t\}}]\label{2div} \tag{IV}\\
&\phantom{=} - \E [F(\Phi_t(\u_0 + \v_1; \xi+h_{\v_1}^M))\mathcal E(h_{\v_1}^M) \1_{\{\tau^M_{\v_1} < t\}}]\label{2dv} \tag{V}.
\end{align*}
We have that 
\begin{equation} \label{2dtausmall} 
|\eqref{2di}| \le \|F\|_{\infty} \E [\mathcal E(h_{\v_0}^M) \1_{\{\tau^M_{\v_0} < t\}}],\hspace{15pt}
|\eqref{2dv}| \le \|F\|_{\infty} \E [\mathcal E(h_{\v_1}^M) \1_{\{\tau^M_{\v_1} < t\}}],
\end{equation}
by Lemma \ref{lemma:hdef}, (i), and recalling that $\E[\mathcal E(h^M_{\v_j})] = 1$, 
\begin{equation}\label{2diiest}
\begin{aligned}
|\eqref{2dii}| &=  \E [F(\Phi_t(\u_0 + \v_0; \xi+h_{\v_0}))\mathcal E(h_{\v_0}^M) \1_{\{\tau^M_{\v_0} \ge t\}} - F(\Phi_t(\u_0; \xi))\mathcal E(h_{\v_0}^M) \1_{\{\tau^M_{\v_0} \ge t\}}]\\
&= \E\big[ \big(F(\Phi_t(\u_0 + \v_0; \xi+h_{\v_0})) - F(\Phi_t(\u_0; \xi))\big) \mathcal E(h_{\v_0}^M) \1_{\{\tau^M_{\v_0} \ge t\}}\big]\\
&\le e^{-\frac t 4} \|F\|_{\H^{1-\eps}-\Lip} \|\v_0\|_{\H^{1-\eps}},
\end{aligned}
\end{equation}
and similarly
\begin{equation}\label{2divest}
|\eqref{2div}| \le e^{-\frac t 4} \|F\|_{\H^{1-\eps}-\Lip} \|\v_1\|_{\H^{1-\eps}}. 
\end{equation}
Finally, by \eqref{Ediffestimate}, for every $\eta > 0$, 
\begin{equation}\label{2diiiest}
\begin{aligned}
|\eqref{2diii}| &\le \|F\|_{\infty} \E|\mathcal E(h_{\v_0}^M) \1_{\{\tau^M_{\v_0} \ge t\}} - \mathcal E(h_{\v_1}^M) \1_{\{\tau^M_{\v_1} \ge t\}}|\\
&\le \E[\mathcal E(h_{\v_0}^M) \1_{\{\tau^M_{\v_0} \ge t\}}] +\E[\mathcal E(h_{\v_1}^M) \1_{\{\tau^M_{\v_1} \ge t\}}] \\
&\phantom{\le}- 
\eta(\P(\{\mathcal E(h_{\v_0}^M) \ge \eta\} \cap \{\tau^M_{\v_0} \ge t\}) + \P(\{\mathcal E(h_{\v_1}^M) \ge \eta\} \cap \{\tau^M_{\v_1} \ge t\}) -1).
\end{aligned}
\end{equation}
Putting \eqref{2dtausmall}, \eqref{2diiest}, \eqref{2divest}, and \eqref{2diiiest} together, and by symmetry between $\v_0$ and $\v_1$, we obtain 
\begin{equation}
\begin{aligned}
&\phantom{\les} \E |F(\Phi_t(\u_0 + \v_0; \xi)) - F(\Phi_t(\u_0+\v_1; \xi))| \\
&\le \|F\|_{\infty}\big( \E [\mathcal E(h_{\v_0}^M) \1_{\{\tau^M_{\v_0} < t\}}] + \E [\mathcal E(h_{\v_1}^M) \1_{\{\tau^M_{\v_1} < t\}}] + \E[\mathcal E(h_{\v_0}^M) \1_{\{\tau^M_{\v_0} \ge t\}}] +\E[\mathcal E(h_{\v_0}^M) \1_{\{\tau^M_{\v_0} \ge t\}}]\\
&\phantom{\le \|F\|_{\infty}\big( )} -\eta(\P(\{\mathcal E(h_{\v_0}^M) \ge \eta\} \cap \{\tau^M_{\v_0} \ge t\}) + \P(\{\mathcal E(h_{\v_1}^M) \ge \eta\} \cap \{\tau^M_{\v_1} \ge t\}) -1) \big) \\
&\phantom{\le} + \|F\|_{\H^{1-\eps}-\Lip} e^{-\frac t4} (\|\v_0\|_{\H^{1-\eps}} + \|\v_1\|_{\H^{1-\eps}})\\
&\le \|F\|_{\infty} (2 - \eta(\P(\{\mathcal E(h_{\v_0}^M) \ge \eta\} \cap \{\tau^M_{\v_0} \ge t\}) + \P(\{\mathcal E(h_{\v_1}^M) \ge \eta\} \cap \{\tau^M_{\v_1} \ge t\}) -1) \\
&\phantom{\le}+ \|F\|_{\H^{1-\eps}-\Lip} e^{-\frac t4} (\|\v_0\|_{\H^{1-\eps}} + \|\v_1\|_{\H^{1-\eps}}).
\end{aligned}
\end{equation}
Therefore, by taking limits for $M \to \infty$, we obtain \eqref{pphi2asfestimate} with 
\begin{gather}
\eps_0 = 1 - \frac 12 \sup_{\eta > 0} \eta \inf_{t>0} \limsup_{M\to \infty} (\P(\{\mathcal E(h_{\v_0}^M) \ge \eta\}) + \P(\{\mathcal E(h_{\v_1}^M) \ge \eta\}) -1). 
\end{gather}
By Chebishev, for $0 < \eta < 1$ we have that 
\begin{align*}
\P(\{\mathcal E(h_{\v_j}^M) < \eta\}) &= \P\Big(\Big\{ \frac12 \int_0^t \norm{h(t')}_{L^2}^2 - \int_0^t \dual{h(t')}{\xi}_{L^2} > \log \eta^{-1} \Big\}\Big)\\
&\les \frac{\E\Big[\int_0^t \norm{h_{\v_j}^M(t')}_{L^2}^2\Big]}{\log \eta^{-1}} + \frac{\E\Big|\int_0^t \dual{h_{\v_j}^M(t')}{\xi}_{L^2}\Big|^2}{(\log \eta^{-1})^2} \\
&\les \frac{\E\|h_{\v_j}\|_{L^2}^2}{\log \eta^{-1}}
\end{align*}
Therefore, by Lemma \ref{lemma:hdef}, \rm{(ii)}, choosing 
$$ \eta = \exp\Big(- C(\u_0) (1 + \|v_0\|_{\H^{1-\eps}}^{2k-2} +\|v_1\|_{\H^{1-\eps}}^{2k-2})^2(\|v_0\|_{\H^{1-\eps}} + \|v_1\|_{\H^{1-\eps}})^2\Big)  $$
for some $C(\u_0) \gg \E|K(\u_0, \xi)|$, we obtain that $\P(\{\mathcal E(h_{\v_j}^M) \ge \eta\}) \ge \frac 34$, and so $\eps_0(\u_0, \v_0, \v_1) < 1$. 

We now move to the case $\v_0 \to 0$, $\v_1 = 0$. From \eqref{2dtausmall}, \eqref{2diiest}, \eqref{2divest}, and \eqref{2diiiest}, we obtain that 
\begin{align*}
&\E [F(\Phi_t(\u_0 + \v_0; \xi)) - F(\Phi_t(\u_0; \xi))] \\
&\le \|F\|_{\infty}\big( \E [\mathcal E(h_{\v_0}^M) \1_{\{\tau^M_{\v_0} < t\}}] + \E|\mathcal E(h_{\v_0}^M) \1_{\{\tau^M_{\v_0} \ge t\}} - 1|\big) \\
&\phantom{\le} + \|F\|_{\H^{1-\eps}-\Lip} e^{-\frac t4} \|\v_0\|_{\H^{1-\eps}}.
\end{align*}
Therefore, we can choose 
\begin{equation} \label{2deps0fel}
\eps_0(\u_0, \v_0, 0) = \frac12 \sup_{t>0} \inf_{M > 0}  \E [\mathcal E(h_{\v_0}^M) \1_{\{\tau^M_{\v_0} < t\}}] + \E|\mathcal E(h_{\v_0}^M) \1_{\{\tau^M_{\v_0} \ge t\}} - 1|
\end{equation}
Fixing $M > 0$, by Lemma \ref{lemma:hdef}, \rm{(ii)}, and by the fact that $\E|[\mathcal E(h_{\v_0}^M)|^2 \le C(M)$, we obtain that 
$$ \lim_{\|\v_0\|_{\H^{1-\eps}} \to 0} \E|\mathcal E(h_{\v_0}^M)-1| = 0.$$
Therefore, by \eqref{2deps0fel}, by definition of $\tau_{\v_0}^M$ and by Lemma \ref{lemma:hdef}, \rm{(ii)} again,
\begin{align*}
&\limsup_{\|\v_0\|_{\H^{1-\eps}} \to 0} \eps_0(\u_0, \v_0, 0) \\
&\le \frac12 \limsup_{\|\v_0\|_{\H^{1-\eps}} \to 0} \sup_{t > 0} \E [\mathcal E(h_{\v_0}^M) \1_{\{\tau^M_{\v_0} < t\}}] + \E|\mathcal E(h_{\v_0}^M) \1_{\{\tau^M_{\v_0} \ge t\}} - 1| \\
& \le \limsup_{\|\v_0\|_{\H^{1-\eps}} \to 0} C(M)^\frac12 \P(\{\tau_{\v_0}^M < \infty\})^\frac12 \\
& \le \limsup_{\|\v_0\|_{\H^{1-\eps}} \to 0} C(M)^\frac12 \P(\|h_{\v_0}\|_{L^2_{t,x}}^2 \ge M \})^\frac12 \\
& = 0.
\end{align*}
\end{proof}

\subsection{Restricted asymptotic strong Feller and restricted coupling properties}
In view of Proposition \ref{pphi2asfprep}, we consider the space 
\begin{equation}\label{Xdef2d}
X := Y + \H^{1-\eps} \cup \{\infty\},
\end{equation}
with the distance 
\begin{equation}\label{Xd2d}
d_X(\u_0,\u_1) = \begin{cases}
\min(\|\u_0-\u_1\|_{\H^{-\eps}},1) &\text{ if } \u_0,\u_1 \neq \infty, \\
1 &\text{ if } \u_0 = \infty, \u_1 \neq \infty \text{ or } \u_0 \neq \infty, \u_1 = \infty, \\
0 & \text{ if } \u_0 = \u_1 = \infty.
\end{cases}
\end{equation}
In view of \eqref{Ydef1}, \eqref{Gamma2d}, and Proposition \ref{wickproperties}, 3., we have that $X \subseteq \Gamma \cup \{\infty\}$. 
Therefore, the flow $\Phi_t$ is well defined for every $\u_0 \in X$. Moreover, by the decomposition 
$$ \Phi_t(\u_0,\xi) = S(t) \u_0 + \ve{\psi}[\xi](t) + \v(t), $$
and the fact that $S(t)$ is bounded on $\H^{1-\eps}$, we have that $\Phi_t(X,\xi) \subseteq X$ for every $t \ge 0$ a.s. 
We want to use the flow to define a Markov semigroup $P_t$ on $X$ so that Assumptions 1,2 are satisfied. Notice that, because of the implicit definition of the space $X$ (or more precisely, of the space $Y$ in \eqref{Ydef1}), we have no guarantee a priori that  the set $X \subset H^{-\eps}$ is Borel measurable.\footnote{More precisely, while it is possible to show that the set $Y$ is Borel-measurable by adapting the arguments of Remark \ref{Ptborel}, it is unclear if the set $Y + \H^{1-\eps}$ remains measurable.} However, if $\ph: X \to \R$ is a Borel function, there exists $\wt\ph: \Gamma \cup \{\infty\} \to \R$ Borel so that $\wt\ph|_{X} = \ph$.\footnote{For indicator functions of open and closed sets, this follows from the fact that closed and open sets in $X$ are the intersection of a (respectively) closed or open set in $\Gamma \cup \infty$ with $X$. For indicator functions of Borel sets, we obtain this from the definition of the Borel $\sigma$-algebra as the smallest $\sigma$-algebra that contains open sets, together with the previous step and the fact that the intersection of a $\sigma$-algebra with a set is a $\sigma$-algebra on that set. Finally, for a general Borel function, we obtain the result by writing it as the limit of simple functions, and extending every single function to $\Gamma \cup \{\infty\}$. Notice that the set where the sequence of simple functions converges is measurable, and will contain $X$ by construction. We will use this fact (and more generally, the consequences of this construction) liberally throughout the rest of the paper.} Moreover, by invariance of $X$, we have that for every $\u_0 \in X$,   
\begin{equation*}
\E[\ph(\Phi_t(\u_0,\xi))] = \E[\wt\ph(\Phi_t(\u_0,\xi))].
\end{equation*}
Therefore, for $s \in \R$, 
\begin{equation*}
\{u_0 \in X: \E[\ph(\Phi_t(\u_0,\xi))] \le s\} = \{\u_0 \in \Gamma \cup \{\infty\}: \E[\wt\ph(\Phi_t(\u_0,\xi))] \le s \} \cap X,
\end{equation*}
which is the intersection of a Borel set in $\Gamma \cup \{\infty\}$ (see Remark \ref{Ptborel}) and $X$, so it is Borel in $X$. 
Therefore, we can define
\begin{equation}\label{semiX2d}
P_t \ph(\u_0) := \E[\ph(\Phi_t(\u_0,\xi))].
\end{equation}
In view of Lemma \ref{2dcompatibility}, we have that $P_t$ is a Markov semigroup on $\mathscr L^\infty(X)$, thus the space $X$ together with the semigroup $P_t$ satisfies Assumptions \ref{ass:metric}, \ref{ass:markov}. In order to put ourselves in the framework of Section \ref{theory}, we define 
$$\G = \H^{1-\eps},$$
with 
\begin{equation} \label{Gdist2d}
 |\v_0| := \min(\| \v_0 \|_{\H^{1-\eps}}, 1), 
\end{equation}
and for $\u_0 \in X$, $\v_0 \in \G = \H^{1-\eps}$, 
\begin{equation}\label{Gact2d}
\tau_{\v_0}(\u_0) := 
\begin{cases}
\u_0 + \v_0 & \text{ if } \u_0 \neq \infty\\
\infty & \text{ if } \u_0 = \infty.
\end{cases}
\end{equation}
It easy to check that Assumptions \ref{ass:metric}--\ref{ass:tautopology} hold, with the possible exception of the measurability condition \eqref{measurability}. In order to show that $\tau(B_r \times K)$ is measurable, we will need the following lemma.
\begin{lemma}\label{closedstability}
Let $C \subseteq \H^{-\eps}$ be a closed set in $\H^{-\eps}$. Then for every $r \ge 0$, 
\begin{equation}
C + \{ \v \in \H^{1-\eps}: \|\v\|_{\H^{1-\eps}} \le r \}
\end{equation}
is closed in $\H^{-\eps}$. 
\end{lemma}
\begin{proof}
Since $\H^{-\eps}$ is a metric space, it is enough to show that if $(x_n)_{n \in \N}$ is a sequence in $C$, $\|\v_n\|_{\H^{1-\eps}} \le r$, and
$$ \lim_{n \to \infty} x_n + \v_n = x, $$
then $x \in C + \{ \v \in \H^{1-\eps}: \|\v\|_{\H^{1-\eps}} \le r \}.$ By compactness of Sobolev embeddings, the ball $\{ \v \in \H^{1-\eps}: \|\v\|_{\H^{1-\eps}} \le r \}$ is compact in $\H^{-\eps}$, so there exists a subsequence $\v_{n_k}$ such that for some $\v \in \H^{-\eps}$,
$$ \lim_{k \to \infty} \v_{n_k} = \v, $$
and $\| \v \|_{\H^{1-\eps}} \le r$. Therefore, 
$$ \lim_{n \to \infty} x_{n_k} = x - \v. $$
Since $C$ is closed, we have that $x - \v \in C$. Therefore, 
$$ x = \overbrace{x-\v}^{\in C} +\hspace{5pt} \v \hspace{5pt} \in  C + \{ \v \in \H^{1-\eps}: \|\v\|_{\H^{1-\eps}} \le r \}.$$
\end{proof}
For a compact set $K$, we have that
\begin{equation*}
\tau(B_r \times K) =
\begin{cases}
K + \{ \v \in \H^{1-\eps}: \|\v\|_{\H^{1-\eps}} \le r \} & \text{ if } r < 1, \\
K + \H^{1-\eps} = \bigcup_{n \in \N} K + \{ \v \in \H^{1-\eps}: \|\v\|_{\H^{1-\eps}} \le n \} & \text{ if } r \ge 1.
\end{cases}
\end{equation*}
Since $K$ is compact as a subset of $X$, then it is also compact (hence closed) as a subset of $\H^{-\eps}$. Therefore, by Lemma \ref{closedstability}, $\tau(B_r \times K)$ is a closed set (for $r < 1$) or a countable union of closed sets (for $r\ge 1$) in $\H^{-\eps}$ hence it is Borel. Since $\tau(B_r \times K)\subseteq X$, then $\tau(B_r \times K)$ is Borel measurable as a subset of $X$. 

In conclusion, we have that $X$ defined in \eqref{Xdef2d} with the distance \eqref{Xd2d}, the semigroup $P_t$ defined in \eqref{semiX2d} and the group $\G = \H^{1-\eps}$ with the absolute value \eqref{Gdist2d} and the action \eqref{Gact2d} satisfy Assumptions \ref{ass:metric}--\ref{ass:tautopology}. 
Moreover, we have the following.
\begin{proposition}\label{pphi2rasf}
The semigroup $P_t$ on $\mathscr L^\infty(X)$ has the following properties.
\begin{enumerate}
\item The semigroup $P_t$ has the asymptotic strong Feller property restricted to the action of $\H^{1-\eps}$ on the set $Y$.
\item The semigroup $P_t$ has the asymptotic coupling property restricted to the action of $\H^{1-\eps}$ on the set $X$, with $r(\u_0) = \infty$ for every $\u_0 \in X$.
\end{enumerate}
\end{proposition}
\begin{proof}
Both statements follow from the estimates \eqref{pphi2asfestimate} and \eqref{pphi2cont} and the respective definitions, except for the technical issue that Proposition \ref{pphi2asfprep} holds for functions $\wt F$ which are Borel-measurable in $\Gamma \cup \{\infty\}$, 
while we need the estimates to hold true for functions $F$ which are Borel-measurable in $X$. 

Therefore, the proposition is proven if we show that for every 
$F: X \to \R$ Borel measurable with $\|F\|_{\infty} < \infty$, $\|F\|_{\H^{1-\eps}-\Lip} < \infty$, there exists $\wt F: \Gamma \cup \{\infty\} \to \R$ Borel measurable such that $\|\wt F\|_{\infty} \le \|F\|_{\infty}$, $\|\wt F\|_{\H^{1-\eps}-\Lip}\le \|F\|_{\H^{1-\eps}-\Lip}$, and that satisfies
\begin{equation}\label{wtFcompatibility}
\wt{F}(\u_0) = F(\u_0) \text{ for every } \u_0 \in X.
\end{equation}
Since $F$ is a measurable function, there exists an extension $\overline{F}: \Gamma \cup \{\infty\} \to \R$ which is measurable, $\|\overline{F}\|_\infty = \|F\|_{\infty}$, and satisfies 
$\overline{F}(\u_0) = F(\u_0)$ for every $\u_0 \in X$.  Define 
\begin{equation}
 \wt F(\u_0) := \inf_{\v \in \H^{1-\eps}} \overline{F}(\u_0 + \v) +\|F\|_{\H^{1-\eps}-\Lip} \|\v\|_{\H^{1-\eps}}.
\end{equation}
It is easy to check that $-\|\overline F\|_{\infty} \le \wt{F}(\u_0) \le \overline F(\u_0)$, so $\|\wt F\|_{\infty} \le \|\overline{F}\|_{\infty} = \|F\|_{\infty}$, and that $\|\wt F\|_{\H^{1-\eps}-\Lip} \le \|F\|_{\H^{1-\eps}-\Lip} $. Therefore, if $\v_n$ is a countable dense subset of $\H^{1-\eps}$, we also have that 
$$ \wt F(\u_0) := \inf_{n \in \N} \overline{F}(\u_0 + \v_n) +\|F\|_{\H^{1-\eps}-\Lip} \|\v_n\|_{\H^{1-\eps}}, $$
so $\wt F$ is also measurable. Finally, for $\u_0 \in X$, since $\u_0 + \v \in X$ as well for every $\v \in \H^{1-\eps}$, we have that 
\begin{align*}
\overline F(\u_0) = F(\u_0) & \ge \wt F(\u_0) \\
& = \inf_{\v \in \H^{1-\eps}} F(\u_0 + \v) +\|F\|_{\H^{1-\eps}-\Lip} \|\v\|_{\H^{1-\eps}} \\
& \ge F(\u_0) + \inf_{\v \in \H^{1-\eps}} - |F(\u_0 + \v) - F(\u_0)| +\|F\|_{\H^{1-\eps}-\Lip} \|\v\|_{\H^{1-\eps}} \\
& \ge F(\u_0) +  \inf_{\v \in \H^{1-\eps}}  - \|F\|_{\H^{1-\eps}-\Lip} \|\v\|_{\H^{1-\eps}} +\|F\|_{\H^{1-\eps}-\Lip} \|\v\|_{\H^{1-\eps}} \\
& = F(\u_0),
\end{align*}
so we obtain \eqref{wtFcompatibility}.
\end{proof}
\begin{remark}
In what follows, we are never going to use the fact that the semigroup has the $\mathrm{(rASF)_Y}$ property, and we will only focus on the consequences of the $\mathrm{(rAC)_X}$ property. The reason is self-evident from the support statements of Theorem \ref{supp1} and \ref{supp2}, that require the relevant property to hold on the whole space. 

If on top of the $\mathrm{(rAC)_X}$ property, we also had the $\mathrm{(rASF)_{X}}$ property, it would actually be possible to show that the flow is globally well posed on $X\setminus \{\infty\}$, in the sense that for every $\u_0 \in X$, $\u_0 \neq \infty$, then for every $t \ge 0$, that $\Phi_t(\u_0,\xi) \neq \infty$ a.s. 
This would provide a global well posedness statement for every initial data of the form $\u_0 \in Y + \H^{1-\eps}$. Unfortunately, the techniques of this paper are not strong enough to provide such a result (and it is unclear if it is true to begin with). The technical reason is that, in the construction of the Girsanov shift of Lemma \ref{lemma:hdef}, we have no way to control exponential moments of $h$ when $\|\v_0\|_{\H^{1-\eps}} \gg 1$. 
Indeed, it is in principle possible that for initial data not on the set $Y$, the solution blows up with positive probability. What the Girsanov shift guarantees, that is encoded in the $\mathrm{(rAC)_X}$ property, is that the solution starting from $\u_0 + \v_0$ will follow (in law) the trajectory of the flow starting from  $\u_0$ with strictly positive probability.

Nevertheless, we decided to include the definition of $\mathrm{(rASF)_S}$ property (and to show it on the set $Y$) in order to make it is easier to draw comparisons with the existing theory.  
\end{remark}

\subsection{Ergodicity of the $P(\Phi)_2$ measure}
We now move to proving the ergodicity of the measure $\rho$. 
In order to be able to apply the theory described in Section \ref{theory}, we need to a procedure to associate to a measure $\mu$ defined on $\H^{-\eps}$ a measure $\io^\ast \mu$ defined on $X$. 
We recall that from the definition \eqref{Xdef2d}, it is not clear if the space $X$ is measurable, so we cannot simply define $\io^\ast\mu$ as the restriction of $\mu$ to $X$. However, we have the following.
\begin{lemma}\label{measurerestriction}
Let $\mu$ be a finite, nonnegative measure on $\Bo(\H^{-\eps})$. Then there exists a unique Radon measure $\io^\ast\mu$ such that for every compact set $K \subseteq X$, 
\begin{equation}\label{restrictionvalue}
 \io^\ast \mu(K) = \mu(K). 
\end{equation}
Moreover, if $\mu$ is concentrated in $\wt X = \Gamma \cup \{\infty\}$ and it is an invariant probability measure in the sense that \eqref{invariancedef2d} holds, and $\io^\ast\mu (X) =1$, then 
$\io^\ast\mu$ is an invariant probability measure for $P_t$ defined in \eqref{semiX2d}.
\end{lemma}
\begin{proof}
Since Radon measures are uniquely determined by their values on compact sets, uniqueness follows from \eqref{restrictionvalue}, so we just need to show existence. Let 
$$ \lambda := \sup_{K \subseteq X \text{ compact}} \mu(K)\le \mu(\H^{-\eps}) < \infty,$$
and let $K_n \subseteq X$ be an increasing sequence of compact sets so that $\lim_{n \to \infty} \mu(K_n) = \lambda$. For $E \subseteq X$ Borel, we define 
$$ \io^\ast\mu(E) := \lim_{n \to \infty} \mu(E \cap K_n). $$
Notice that since $E$ is a Borel set in $X$, there exists a Borel set $E' \subseteq \H^{-\eps}$ so that 
$$ E = E' \cap X, $$
so $E \cap K_n = E' \cap K_n$ is a Borel set in $\H^{-\eps}$, and $\io^\ast\mu(E)$ is well defined. By monotone convergence, $\io^\ast\mu$ is $\sigma$-additive, so $\io^\ast\mu$ actually defines a measure. We now check that it is indeed a Radon measure. Recalling that finite measures on a compact metric space are Radon, it is enough to show tightness, or more specifically that   
$$ \lim_{n \to \infty} \io^\ast\mu(X\setminus K_n) = 0.$$
Noticing that by definition $\io^\ast\mu(K_n) = \mu(K_n)$, we have that 
\begin{align*}
 \lim_{n \to \infty} \io^\ast\mu(X\setminus K_n) &= \lim_{n \to \infty} \big(\io^\ast\mu(X) - \io^\ast\mu(K_n)\big) 
 = \lim_{m \to \infty} \mu(K_m) -  \lim_{n \to \infty} \mu(K_n)
 = 0,
\end{align*}
so $\io^\ast\mu$ is a Radon measure. Finally, we show \eqref{restrictionvalue}. For $K \subseteq X$ compact, recalling the definition of  $\lambda$, we have that 
\begin{align*}
0 \le \lim_{n \to \infty} \mu(K \setminus K_n) 
= \lim_{n \to \infty} \mu(K \cup K_n) - \mu(K_n) 
\le \lambda - \lim_{n \to \infty} \mu(K_n) 
= 0.
\end{align*}
Therefore,
\begin{align*}
\io^\ast\mu(K) = \lim_{n \to \infty} \mu(K \cap K_n) = \lim_{n \to \infty} \mu(K) - \mu(K \setminus K_n) = \mu(K).
\end{align*}

We now assume \eqref{invariancedef2d} and that $\io^\ast\mu(X) = 1$, and show that $\io^\ast\mu$ is invariant for $P_t$. We first show that 
for every function $f \in L^\infty(\H^{-\eps})$, 
\begin{equation}\label{compatibilitymuiota}
 \int f (\u_0) d \io^\ast\mu(\u_0) = \int f(\u_0) d \mu(\u_0). 
\end{equation}
Since simple functions are dense in $\mathscr L^\infty(\H^{-\eps})$, it is enough to show that this holds for $f = \1_{E'}$, 
where $E'$ is a Borel subset of $\H^{-\eps}$. In this setting, \eqref{compatibilitymuiota} reduces to showing that 
$$ \io^\ast \mu(E\cap X) = \mu (E).$$
By definition of $\io^\ast\mu$, we have that $\io^\ast \mu(E\cap X) = \lim_n \mu(E\cap K_n)$. Moreover, since $\io^\ast\mu(X) = 1$, 
we have that $\lim_{n\to\infty} \mu(K_n^c) = 0$. Therefore,
\begin{align*}
\mu (E) &= \lim_{n\to\infty} \mu(E\cap K_n) + \mu(E\setminus K_n) \\
&= \lim_{n \to \infty} \mu(E\cap K_n)\\
&= \io^\ast \mu(E\cap X),
\end{align*}
so we have \ref{compatibilitymuiota}. In order to prove invariance, fix $\ph \in \mathscr L^\infty(X)$, and let $\wt \ph \in \mathscr L^\infty(\H^{-\eps})$ be such that $\wt \ph|_{X} = \ph$. By \eqref{semiX2d}, \eqref{compatibilitymuiota}, and \eqref{invariancedef2d}, we have that
\begin{align*}
\int P_t\ph (\u_0) d \io^\ast\mu(\u_0) &= \int \E[\wt{\ph} (\Phi_t(\u_0,\xi))] d \io^\ast\mu(\u_0) \\
&= \int \E[\wt{\ph} (\Phi_t(\u_0,\xi))] d \mu(\u_0) \\
&= \int \wt{\ph}(\u_0) d \mu(\u_0) \\
&= \int \ph (\u_0) d \io^\ast\mu(\u_0).
\end{align*}
\end{proof}
The next lemma, even if technically fairly simple, is one of the central pieces of the proof strategy for Theorem \ref{uniqueness2d}. It essentially states that the support Theorem \ref{supp2} is enough to obtain a contradiction further down the line. In our setting, the statement should be interpreted as a kind of ``irreducibility under the action of $\H^{1-\eps}$" for the Gaussian measure $\rho_0$ (and as a consequence, for the invariant measure $\rho$).  
\begin{lemma}\label{0-1}
Let $\pi: X \to X/\H^{1-\eps}$ be the canonical projection, and consider the measure $\pi_\#\io^\ast\rho_0$, where $\rho_0$ is the gaussian measure \eqref{rho0}. For every measurable set $E \in \pi_\#\Bo(X)$, we have that 
\begin{equation*}
\pi_\#\io^\ast\rho_0(E) = 0 \hspace{10pt}\text{ or }\hspace{10pt} \pi_\#\io^\ast\rho_0(E) = 1.
\end{equation*}
\end{lemma}
\begin{proof}
We start by showing this property for the measure $\pi'_\#\rho_0$ instead of $\pi_\#\io^\ast\rho_0$, where 
$\pi': \H^{-\eps} \to \H^{-\eps}/ \H^{1-\eps}$ denotes the canonical projection. 
By \eqref{rho0law}, we can see $\mu$ as the law of the random variable $\U = (U, V)$, with 
\begin{align*}
U &= \frac{1}{2\pi}\Re{\Big(\sum_{n \in \Z^2} \frac{g_n}{\jb n} e^{i n \cdot x}\Big)},\\
V &= \frac{1}{2\pi}\Re{\Big(\sum_{n \in \Z^2} {h_n} e^{i n \cdot x}\Big)},
\end{align*}
where $g_n$, $h_n$, are i.i.d., centred, \ complex valued gaussian random variables, with $\E g_n^2 = \E h_n^2 = 0$, $\E |g_n|^2 = \E |h_n|^2 = 1$. 
Call $\U_{>N} = (U_{>N},V_{>N})$, with  
\begin{align*}
U_{>N} &= \frac{1}{2\pi}\Re{\Big(\sum_{n \in \Z^2, |n|_\infty > N} \frac{g_n}{\jb n} e^{i n \cdot x}\Big)},\\
V_{>N} &= \frac{1}{2\pi}\Re{\Big(\sum_{n \in \Z^2, |n|_\infty > N} {h_n} e^{i n \cdot x}\Big)},
\end{align*}
and let $\U_{\le N} = \pi_N\U = \U - \U_{>N}$. Since $e^{in\cdot x} \in \H^{1-\eps}$ for every $n \in \Z^2$, we have that $\U_{\le N} \in \H^{1-\eps}$.
By definition, for a set $E \in \pi'_\#\Bo(\H^{-\eps})$, we have  
$$\pi'_\#\rho_0(E) = \P(\U \in (\pi')^{-1}(E)). $$
Moreover, since $(\pi')^{-1}(E) = (\pi')^{-1}(E) + \H^{1-\eps}$, we have the equivalence 
$$\U \in (\pi')^{-1}(E)  \iff \U_{>N} \in (\pi')^{-1}(E).$$
Therefore, 
$$\{\U \in (\pi')^{-1}(E)\} = \{\U_{>N} \in (\pi')^{-1}(E)\} . $$
This shows that for every $n \in \N$, the event  $\{\U \in (\pi')^{-1}(E)\}$ belongs to the $\sigma$-algebra generated by $\{g_n, h_n: n > |n|_{\infty}\}$. By Kolmogorov's zero-one law, this implies that $\P(\{\U \in (\pi')^{-1}(E)\}) = 0$ or $\P(\{\U \in (\pi')^{-1}(E)\}) = 1$. Let now $F \in \pi_\#\Bo(X/\H^{-\eps})$. By definition, this means that there exists a set $\wt F\in \Bo(\H^{-\eps})$ such that 
\begin{equation}\label{setrestrictiondef}
 \pi^{-1}(F) = \wt F \cap X, \quad  \pi^{-1}(F^c) = \wt F^c \cap X.
\end{equation}
By Lemma \ref{lem:Xfull2d}, $\rho_0(Y)=1$. Therefore, there exists a $\sigma$-compact set $\wt Y\subseteq Y$ such that $\rho_0(\wt Y) = 1$ as well. In particular, by Lemma \ref{closedstability}, $\wt Y + \H^{1-\eps} \in \Bo(\H^{-\eps})$, and clearly $\rho_0(\wt Y + \H^{1-\eps})=1$ as well. Therefore, by \eqref{compatibilitymuiota},
\begin{align}\label{piiotacompatibility}
\io^\ast\mu( \pi^{-1}(F)) &= \io^\ast \mu(\wt F\cap X) = \mu(\wt F) = \mu(\wt F \cap (\wt Y + \H^{1-\eps})).
\end{align}
Morever, by \eqref{setrestrictiondef}, since $\wt Y + \H^{1-\eps} \subseteq Y + \H^{1-\eps} \subseteq X$,
\begin{align*}
\wt F \cap (\wt Y + \H^{1-\eps}) = (\wt F \cap X) \cap (\wt Y + \H^{1-\eps}) =  \pi^{-1}(F) \cap (\wt Y + \H^{1-\eps})  = \pi^{-1}(F \cap \pi(\wt Y)).
\end{align*}
Notice that, since $X\setminus \{\infty\}$ is invariant under the action of $\H^{1-\eps}$ over $\H^{-\eps}$, for every set $A \subset X/\H^{1-\eps}$, we have $\pi^{-1}(A) = (\pi')^{-1}(A)$. Therefore, by the the first part of the proof and \eqref{piiotacompatibility}, we obtain that 
\begin{equation*}
\io^\ast\mu( \pi^{-1}(F)) = \mu(\wt F \cap (\wt Y + \H^{1-\eps})) =  \mu(\pi^{-1}(F \cap \pi(\wt Y))) =  \mu((\pi')^{-1}(F \cap \pi(\wt Y))) = 0 \text{ or }1.
\end{equation*}

\end{proof}
We would like to point out that, in the space $X/\H^{1-\eps}$, measures that take only values $0$ or $1$ do not necessarily correspond to measures concentrated in one point. 
In the particular case of $\pi_\#\io^\ast\rho_0$, we actually have that for every $y \in X/\H^{1-\eps}$, $\pi_\#\io^\ast\rho_0(\{y\}) = 0$. Indeed, if $\U_1, \U_2$ are two independent copies of $\U$, we have that $\Law(\U_1 - \U_2) = \Law(\sqrt 2\, \U)$, and 
\begin{align*}
0 &= \P(\{\U \in \H^{1-\eps}\}) = \P(\{\U_1 - \U_2 \in \H^{1-\eps}\}) = \P(\{\pi(\U_1) = \pi(\U_2)\}) \\
&\ge \P(\{\pi(\U_1) = y, \pi(\U_2) = y\}) = \P(\{\pi(\U) = y\})^2 = \pi_\#\io^\ast\rho_0(\{y\})^2.
\end{align*}
Nevertheless, measures that satisfy Lemma \ref{0-1} still share the the following property with Dirac $\delta$ measures.
\begin{lemma}\label{deltaness}
Let $\nu$ be a probability measure on $X/\H^{1-\eps}$ on the $\sigma$-algebra $\pi_\#\Bo$ such that $\nu \ll \pi_\#\io^\ast\rho_0$. Then $\nu = \pi_\#\io^\ast\rho_0$.
\end{lemma}
\begin{proof}
Let $E \in \pi_\#\Bo$. By Lemma \ref{0-1}, $\pi_\#\mu(E)= 0$ or $\pi_\#\mu(E)= 1$. If $\pi_\#\mu(E)= 0$, then $\nu(E) = 0$ as well by absolute continuity. If $\pi_\#\mu(E)= 1$, then $\pi_\#\mu(E^c)= 0$, so $\nu(E^c) = 0$ by absolute continuity, from which we get $\nu(E) = 1$. 
\end{proof}
\begin{proof}[Proof of ergodicity in Theorem \ref{ergodicity2d}]
Suppose by contradiction that the measure $\rho$ is not ergodic. Then there exist $\rho_1, \rho_2 \ll \rho$ with $\rho_1 \perp \rho_2$ and $\rho_1, \rho_2$ are both invariant (in the sense that \eqref{invariancedef2d} holds). 
Since $X \subseteq Y$, by Lemma \ref{lem:Xfull2d} we have that 
$$\io^\ast \rho_1(X) = \io^\ast \rho_2(X) = \io^\ast \rho(X) = 1.$$
Therefore, by Lemma \ref{measurerestriction}, the measures $\io^\ast \rho_1, \io^\ast \rho_2$ are invariant for $P_t$ defined on $X$. Moreover, by \eqref{compatibilitymuiota}, we have $\io^\ast \rho_1 \perp \io^\ast \rho_2$.
By Proposition \ref{pphi2rasf}, $P_t$ has the $\mathrm{(rAC)_X}$ property with $r(\u_0) = \infty$ for every $\u_0 \in X$. Therefore, by Theorem \ref{supp2},
$$\pi_\#\io^\ast \rho_1 \perp \pi_\#\io^\ast \rho_2. $$
However, we have that $\io^\ast \rho_j \ll \io^\ast \rho \ll \io^\ast\rho_0$ for $j=1,2$, 
so by Lemma \ref{deltaness}, $\pi_\#\io^\ast \rho_1 = \pi_\#\io^\ast\rho_0 = \pi_\#\io^\ast \rho_2$, which is a contradiction.
\end{proof}

\subsection{Conditional uniqueness of the $P(\Phi)_2$ measure}
In this final subsection, we are going to derive the uniqueness result of Theorem \ref{uniqueness2d} as a consequence of Theorem \ref{supp2}. The main element of the proof is the following proposition.
\begin{proposition}\label{linearconvergence2d}
Consider the class $W^{1}_{\wick{p}}$ defined in \eqref{pintergrability2d}, and suppose that $\mu \in W^{1}_{\wick{p}}$ is an invariant measure for \eqref{SNLW}. Let $\io^\ast \mu$ be the measure defined in Lemma \ref{measurerestriction}. Finally, let $\pi : X \to X/\H^{1-\eps}$ be the canonical projection, and let $\rho_0$ be the gaussian measure \eqref{rho0}. Then $\io^\ast\mu(X) = 1$, and
$$ \pi_\#\io^\ast\mu = \pi_\#\io^\ast\rho_0. $$
\end{proposition}
In order to be able to show this, we need a couple of preparatory lemmas. 
\begin{lemma}\label{rho0projection}
For every $\u_0 \in \H^{-\eps}$, 
$$ \Law(S(t)\u_0 + \ve \psi[\xi](t)) \rightharpoonup \rho_0 $$
as $t \to \infty$, where the limit is intended as the weak limit of probability measures over $\H^{-\eps}$.
\end{lemma}
\begin{proof}
By the estimate $\|S(t)\u_0\|_{\H^{s}} \les e^{-\frac t2} \|\u_0\|_{\H^{s}},$ we obtain that $\lim_{t \to \infty} S(t)\u_0 = 0$ in $\H^{-\eps}$. Therefore, it is enough to show that 
$$ \Law(\ve \psi[\xi](t)) \rightharpoonup \rho_0. $$
Since both $\ve \psi[\xi](t)$ and $\rho_0$ are Gaussian measures concentrated on $\H^{-\eps}$, it is enough to check that the covariance operator of $\ve \psi[\xi](t)$ converges to the covariance operator of  $\rho_0$ (as trace-class operators over $\H^{-\eps}$). Denoting by $C(t)$ the covariance operator of $\ve \psi[\xi](t)$ and by $C$ the covariance operator of $\rho_0$, 
and writing $[\nabla] := \sqrt{\frac 34-\Delta}$, we have that 
\begin{align*}
C &= 
\begin{pmatrix}
(1 - \Delta)^{-1} & 0 \\
0 & 1
\end{pmatrix},\\
C(t) &= 2\int_0^t e^{-t'} \begin{pmatrix}
\frac{\sin(t' [\nabla])^2}{[\nabla]^2} & \frac{\sin(t' [\nabla])}{[\nabla]}\left(\cos(t'[\nabla]) - \frac 12 \frac{\sin(t' [\nabla])}{[\nabla]}  \right) \\
\frac{\sin(t' [\nabla])}{[\nabla]}\left(\cos(t'[\nabla]) - \frac 12 \frac{\sin(t' [\nabla])}{[\nabla]}  \right) & \left(\cos(t'[\nabla]) - \frac 12 \frac{\sin(t' [\nabla])}{[\nabla]}  \right)^2
\end{pmatrix}dt'\\
&= (1-e^{-t})\begin{pmatrix}
(1 - \Delta)^{-1} & 0 \\
0 & 1
\end{pmatrix} \\
&\phantom{=}+
e^{-t} \begin{pmatrix}
\frac{-2 [\nabla]\sin(2t[\nabla]) + \cos(2t[\nabla]) - 1}{(3-4\Delta)(1-\Delta)}& \frac{\sin(t[\nabla])^2}{[\nabla]^2}\\
\frac{\sin(t[\nabla])^2}{[\nabla]^2} &  \frac{2 [\nabla]\sin(2t[\nabla]) + \cos(2t[\nabla]) - 1}{3-4\Delta}
\end{pmatrix}.
\end{align*}
From these formulas, it is easy to check that $C(t) \to C$ as $ t \to \infty$ in trace class over $\H^{-\eps}$.
\end{proof}
\begin{lemma}\label{lem:piiotacomp}
Let $\mu_1,\mu_2$ be two probability measures on $\Bo(\H^{-\eps})$ such that $\io^\ast \mu_1(X) = \io^\ast \mu_2(X) =1$. Let $\pi' : \H^{-\eps} \to \H^{-\eps}/\H^{1-\eps}$ be the canonical projection. Suppose moreover that $\pi'_\# \mu_1 = \pi'_\#\mu_2$. Then   
$$ \pi_\#\io^\ast \mu_1 = \pi_\#\io^\ast \mu_2. $$
\end{lemma}
\begin{proof}
Since $\io^\ast \mu_1(X) = \io^\ast \mu_2(X) =1$, by definition of the measures $\io^\ast\mu_j$, there exists a $\sigma$-compact set $\wt K \subseteq X$ such that 
$$\mu_1(\wt K) = \mu_2(\wt K) = 1.$$
By Lemma \ref{closedstability}, $\wt K + \H^{1-\eps} \in \Bo(\H^{-\eps})$, and clearly $\mu_j(\wt K + \H^{1-\eps})=1$ as well. Let $E$ be a set in $\pi_\#\Bo(X/\H^{1-\eps})$. Then, by definition of the $\sigma$-algebra $\pi_\#\Bo(X/\H^{1-\eps})$, there exists a set $\wt E \in \Bo(\H^{-\eps} \cup \{\infty\})$ such that 
\begin{equation} \label{piiotacompa}
\pi^{-1}(E) = \wt E \cap X. 
\end{equation}
Therefore, by \eqref{compatibilitymuiota}, 
$$\io^\ast \mu_j \pi^{-1}(E) = \io^\ast \mu_j(\wt E\cap X) = \mu_j (\wt E) = \mu_j (\wt E \cap (\wt K + \H^{1-\eps})).$$
Moreover, by \eqref{piiotacompa},
$$ \wt E \cap (\wt K + \H^{1-\eps}) = (\wt E \cap X) \cap (\wt K + \H^{1-\eps}) = \pi^{-1}(E)  \cap (\wt K + \H^{1-\eps}) = \pi^{-1}(E \cap \pi(\wt K)).$$
Therefore, 
$$ \io^\ast \mu_j \pi^{-1}(E)  = \mu_j (\wt E \cap (\wt K + \H^{1-\eps})) = \mu_j(\pi^{-1}(E \cap \pi(\wt K))) = \pi'_\#\mu_j(E \cap \pi'(\wt K)),$$
and by hypothesis, the last term in the equality does not depend on $j$. Therefore, we obtain that 
$$ \io^\ast \mu_1 \pi^{-1}(E) =  \io^\ast \mu_2 \pi^{-1}(E).  $$
\end{proof}
\begin{proof}[Proof of Proposition \ref{linearconvergence2d}]
Recall the decomposition 
\begin{equation*}
\Phi_t(\u_0,\xi) = S(t)\u_0 + \ve{\psi}[\xi](t) + \v(t),
\end{equation*}
where $\v(t)$ solves the equation \eqref{SNLWv}. Let $\U_0$ be a $\H^{-\eps}$-valued random variable with $\Law(\U_0) = \mu$, and for $t\ge 0$, define 
\begin{align}
\mathbf X(t) &:= S(t)\u_0 + \ve{\psi}[\xi](t),  \label{Xdefrv} \\
\mathbf V(t) &:= \v(t). \label{Vdef2d} 
\end{align}
By invariance of $\mu$, we have that $\Law(\mathbf X(t) + \mathbf V(t)) = \mu$ for every $t \ge 0$. We want to show that (up to subsequences), $\Law(\mathbf X(t),\mathbf V(t))$ has a weak limit as a probability measure over $\H^{-\eps} \times \H^{-\eps}$. 
By Prokhorov's theorem, we just need to show tightness of the couple $(\mathbf X(t),\mathbf V(t))$. By Lemma \ref{rho0projection}, the family ${\Law(\mathbf X(n))}_{n \in \N}$ is tight, so there exists a family of compact sets $K_\dl \subset \H^{-\eps}$ such that 
\begin{equation}\label{tight12d}
 \P(\mathbf X(n) \not\in K_\dl) \le \dl. 
\end{equation}
We now move to the tightness estimate for $\mathbf V$. From \eqref{SNLWv}, we have that $\mathbf V$ solves the equation 
\begin{equation*}
\mathbf V(t) = - \int_0^t S(t-t') \vec{0}{\wick{p(X+V(t'))}} d t',
\end{equation*}
where $X,V$ are respectively the first component of $\mathbf X$ and $\mathbf V$.
From this and H\"older, we obtain that 
\begin{equation*}
\|\mathbf V\|_{\H^{1-\eps}} \les \int_0^t e^{-\frac{t-t'}{2}} \|\wick{p(X+V(t'))}\|_{\H^{-\eps}} dt'
\end{equation*}
Recalling that $\mu \in W^{1}_{\wick{p}}$, and that $\mu$ is invariant, we obtain that 
\begin{equation}\label{Vestimate1}
\E \|\mathbf V(t)\|_{\H^{1-\eps}} \le C(\mu), 
\end{equation}
where 
\begin{equation*}
C(\mu) \sim \int \|\wick{p(u)}\|_{\H^{-\eps}} d\mu(u).
\end{equation*}
Therefore, from Markov's inequality, we obtain that 
\begin{equation}\label{tight22d}
\P(\mathbf V(n) \not\in \{ \|\cdot\|_{\H^{1-\eps}} \le \dl^{-1}C(\mu) \} \le \dl.
\end{equation}
Putting \eqref{tight12d} and \eqref{tight22d} together, we obtain 
\begin{equation}\label{tight32d}
\P\big(\big\{(\mathbf X(n), \mathbf V(n)) \not \in K_{\dl/2} \times \{ \|\cdot\|_{\H^{1-\eps}} \le 2\dl^{-1}C(\mu)\}\big\}\big) \le \dl.
\end{equation}
Since the embedding $\H^{1-\eps} \hookrightarrow \H^{-\eps}$ is compact, this shows tightness for $\Law(\mathbf X(n), \mathbf V(n))$. Therefore, up to subsequences, we have that $\Law(\mathbf X(n),\mathbf V(n)) \rightharpoonup \nu$ as $n \to \infty$, where $\nu$ is a Borel measure on $\H^{-\eps} \times \H^{-\eps}$. Moreover, by \eqref{tight32d}, we have that 
\begin{equation} \label{nusupport2d}
\nu(\H^{-\eps} \times \H^{1-\eps}) = 1.
\end{equation}
We define the map 
$$ \oplus(x,y) := x+y, $$
by invariance of $\mu$, we have that 
\begin{align*}
\oplus_\#\Law(\mathbf X(t),\mathbf Y(t)) = \Law(\mathbf X(t)+\mathbf Y(t)) = \mu.
\end{align*}
Since $\oplus: \H^{-\eps} \times \H^{-\eps} \to \H^{-\eps}$ is continuous, this property passes to limit, and we obtain that 
\begin{equation} \label{nusum2d}
\oplus_\#\nu =\mu.
\end{equation}
Let $\pi': \H^{-\eps} \to \H^{-\eps}/\H^{1-\eps}$ be the canonical projection. We observe that on the set $\H^{-\eps} \times \H^{1-\eps}$, we have that 
\begin{equation*}
\pi'\circ\oplus(x,y) = \pi(x).
\end{equation*}
Therefore, by \eqref{nusupport2d}, \eqref{nusum2d}, and Lemma \ref{rho0projection}, we obtain that 
\begin{align*}
\pi'_\#\mu &= \pi'_\#\oplus_\# \nu \\
&= \pi'_\# \lim_{t \to \infty} \Law( \mathbf X(t))\\
&= \pi'_\# \rho_0.
\end{align*}
In view of Lemma \ref{lem:piiotacomp}, in order to show the analogous statement for $\pi_\#\io^\ast \mu$ and $\pi_\#\io^\ast \rho_0$, we just need to show that $\io^\ast \mu(X) = 1$. Let $\wt K$ be a $\sigma$-compact set such that $\wt K \subseteq Y$ and $\rho_0(K) = 1$, where $Y$ is defined in \eqref{Ydef1}. The existence of such a set follows from Lemma \ref{lem:Xfull2d}. In view of \eqref{nusum2d}, \eqref{nusupport2d} and Lemma \ref{rho0projection}, we have that 
\begin{align*}
\mu( \wt K + \H^{1-\eps}) &= \int \1_{\wt K + \H^{1-\eps}}(\u + \v) d \nu(\u,\v) \\
&\ge \int \1_{\wt K}(\u) \1_{\H^{1-\eps}}(\v) d \nu(\u,\v) \\
&= 1,
\end{align*}
hence $\mu( \wt K + \H^{1-\eps}) = 1.$ Moreover, recalling that the embedding $\H^{1-\eps} \to \H^{-\eps}$ is compact, we have that $\wt K + \H^{1-\eps}$ is a $\sigma$-compact set as well, and $\wt K + \H^{1-\eps} \subseteq X$ by definition of $X$. Therefore, $\io^\ast\mu(X) = 1$.
\end{proof}
We are finally ready to show the uniqueness statement of Theorem \ref{uniqueness2d}.
\begin{proof}[Proof of conditional uniqueness in Theorem \ref{uniqueness2d}]
Our goal is to apply Theorem \ref{supp2}. In order to do so, suppose by contradiction that the $P(\Phi)_2$ measure $\rho$ is not unique in the class $W^{1}_{\wick{p}}$. Let $\mu$ be a invariant measure belonging to $W^{1}_{\wick{p}}$, different from $\rho$. By eventually repeating the decomposition \eqref{singulardecomposition}, we can assume that $\mu$ and $\rho$ are mutually singular. By Lemma \ref{linearconvergence2d}, we have that $\io^\ast\mu(X) = 1$. Therefore, by \eqref{compatibilitymuiota}, we obtain that $\io^\ast\mu \perp \io^\ast\rho$ as well. Moreover, by Lemma \ref{measurerestriction}, $\io^\ast\mu$ is invariant for $P_t$ defined in \eqref{semiX2d}. Since $P_t$ has the restricted coupling property with $r(\u_0) = \infty$ for every $\u_0$ by Proposition \ref{pphi2rasf}, we an apply Theorem \ref{supp2}, and obtain that 
$$ \pi_\#\io^\ast\mu \perp \pi_\#\io^\ast\rho. $$
However, by Lemma \ref{linearconvergence2d}. $ \pi_\#\io^\ast\rho \ll \pi_\#\io^\ast\rho_0 = \pi_\#\io^\ast\mu$, which is a contradiction. 

\end{proof}

\end{document}